\documentclass[a4paper,11pt]{article}

\usepackage[utf8]{inputenc} 
\usepackage{amsfonts}
\usepackage{amssymb}
\usepackage{amsthm}
\usepackage{amsmath}
\usepackage{a4wide}
\usepackage{mathrsfs}
\usepackage{epsfig}
\usepackage{comment}
\usepackage{palatino}
\usepackage{tikz}
\usepackage{fp,ifthen}
\usepackage{esint}
\usepackage{nicefrac}
\usetikzlibrary{decorations.pathreplacing}
\usepackage{float}
\usepackage{dsfont}
\usepackage{enumerate}

\usepackage[a4paper,left=30mm,right=30mm,top=30mm,bottom=30mm,marginpar=20mm]{geometry} 

 \usepackage{lipsum}

\usepackage[colorinlistoftodos,textwidth=2.3cm]{todonotes} 

\newcommand{\pdfgraphics}{\ifpdf\DeclareGraphicsExtensions{.pdf,.jpg}\else\fi}
\usepackage{graphicx}

\usepackage{color}
\definecolor{hanblue}{rgb}{0.27, 0.42, 0.81}
\definecolor{red}{rgb}{1.0, 0.0, 0.0}
\usepackage[colorlinks, citecolor=blue,linkcolor=blue, urlcolor = blue]{hyperref}
\usepackage[english,capitalize]{cleveref}

\usepackage{mathtools}

\theoremstyle{plain}
\newtheorem{theorem}{Theorem}[section]
\newtheorem{conjecture}[theorem]{Conjecture}
\newtheorem{lemma}[theorem]{Lemma}
\newtheorem{proposition}[theorem]{Proposition}
\newtheorem{corollary}[theorem]{Corollary}
\newtheorem{ackn}{Acknowledgements\!}

\theoremstyle{definition}
\newtheorem{definition}[theorem]{Definition}

\theoremstyle{remark}
\newtheorem{remark}[theorem]{Remark}

\numberwithin{equation}{section}

\newcommand{\de}{\ensuremath{\,\mathrm d}} 

\newcommand \eps{\ensuremath{\varepsilon}} 
\renewcommand{\epsilon}{\varepsilon}
\newcommand{\N}{\ensuremath{\mathbb N}}
\newcommand{\R}{\ensuremath{\mathbb R}}
\newcommand{\C}{\ensuremath{\mathbb C}}

\newcommand{\mres}{\mathrel{\tikz[baseline=0ex]{\draw[line width=.5pt,line cap=round] (0ex,1ex) -- (0ex,0ex) -- (1ex,0ex);}}}

\newcommand{\toweakstar}{\overset{*}\rightharpoonup}
\newcommand{\sbullet}{\begin{picture}(1,1)(-0.5,-2.5)\circle*{2}\end{picture}}
\newcommand{\frarg}{\,\sbullet\,}
\newcommand{\ONE}{\mathds{1}}

\begin{document}

\pdfgraphics 

\title{Concentration phenomena and the\\ Vanishing Mass Conjecture}

\author{Luca Gennaioli \footnote{Mathematics Institute, University of Warwick, Coventry CV4 7AL, UK, \href{L.Gennaioli@warwick.ac.uk}{L.Gennaioli@warwick.ac.uk}.} \and Filip Rindler \footnote{Mathematics Institute, University of Warwick, Coventry CV4 7AL, UK, \href{F.Rindler@warwick.ac.uk}{F.Rindler@warwick.ac.uk}.}}

\date{\today}

\maketitle

\vspace{-0.5cm}
\begin{abstract}
\noindent 
Concentrating (that is, non-equi-integrable) sequences of functions satisfying linear PDE constraints arise in a wide variety of problems in PDE, the calculus of variations, and geometric measure theory. In 2001, Bouchitt\'{e} conjectured that such concentrations can always be represented as superpositions of ``simple'' concentrations, a claim he termed the Vanishing Mass Conjecture. We fully resolve this conjecture for all first-order constant-coefficient linear differential operators, proving that the value-distribution (Young) measure of any such sequence admits a Choquet-type decomposition into probability measures whose barycenters lie in the associated Tartar wave cone. In fact, we prove a significantly stronger statement than originally conjectured, namely that the constituent measures are themselves generated by concentrating sequences satisfying the same PDE constraint. This provides a complete structural description of concentrating sequences and yields several applications, including a new proof of a theorem of De~Philippis and the second author on the singular polar of PDE-constrained measures, two types of compensated compactness results, and a surprising result on the support cardinality of extremal concentration Young measures. We also place our results in the context of Morrey's Conjecture, showing that its analogue for pure concentrations is false. Our proofs introduce several new techniques, most notably convexity arguments involving ``barrier functions'' and a careful smoothing procedure via the heat flow that replaces classical Fourier methods.
\end{abstract}

\section{Introduction}

The Vanishing Mass Conjecture, which was introduced in a highly influential paper by Guy Bouchitt\'{e}~\cite{B03} -- for the divergence only and in an ${\rm L}^2$-setting with different scaling\footnote{We learned of the ${\rm L}^1$-version investigated here from Giovanni Alberti, who attributed this version, too, to Bouchitt\'{e}.} -- concerns the precise structure of the admissible asymptotic concentrations in sequences of ${\rm L}^1$-bounded, but not equi-integrable (hence ``concentrating''), functions satisfying a linear PDE constraint. One possible way to state the original conjecture is the following:

\begin{conjecture}
\label{conj:vmc}
Let $(u_j)_j \subseteq {\rm C}^\infty(\Omega;\R^{d\times d})$ be a sequence of (row-wise) divergence-free vector fields such that $u_j \to 0$ in measure and $\|u_j\|_{{\rm L}^1}=1$ for all $j$. Let $\mu$ be a \emph{concentration Young measure} associated to the sequence $(u_j)_j$, that is, any weak* limit point of the sequence
\begin{equation*}
  \mu_j := \bigg(\frac{u_j}{|u_j|}\bigg)_\sharp(|u_j| \, \mathscr L^d\mres\Omega)
  := (|u_j| \, \mathscr L^d\mres\Omega) \circ \bigg(x \mapsto \frac{u_j(x)}{|u_j(x)|}\bigg)^{-1}.
\end{equation*}
Then, there exists a probability measure $\nu_0$ concentrated on the set
\begin{equation} \label{eq:Lambda_div}
\Lambda_{\mathrm{div}}:= \bigl\{ A\in\R^{d\times d} \;:\; {\rm det}\, A=0 \bigr\}
\end{equation}
and a family $(\mu_\xi)_{\xi \in \Lambda_{\mathrm{div}}}$ of probability measures that are supported on the unit sphere in $\mathbb{R}^{d\times d}$ (and weakly*-measurable with respect to $\nu_0$) such that $\mu$ can be written in the form
\[
  \mu=\int\mu_\xi\de\nu_0(\xi)
\]
and for the \emph{barycenters} $[\mu_\xi]$ of the $\mu_\xi$'s it holds that
\[
  [\mu_\xi] := \int \lambda \de \mu_\xi(\lambda) =\xi \qquad\text{for $\nu_0$-a.e.\ $\xi \in \Lambda_{\mathrm{div}}$.}
\]
\end{conjecture}

This statement was originally conceived as a strategy to attack the Optimal Light Structures Conjecture in the field of shape optimization, which concerns the strongest-possible ``light'' structure, as measured by minimal elastic compliance in the limit of the material mass tending to zero. This conjecture goes back to Michell's seminal 1904 work~\cite{M04} on optimal truss design and has attracted a lot of attention since~\cite{KS86_1,KS86_2,KS86_3,AK93_1,AK93_2,AK93_3,A02,O17,O22,BIR23,AIR26}.

Despite its applied origins, the Vanishing Mass Conjecture -- naturally extended to general linear PDE constraints -- has come to be viewed as a fundamental structural principle governing concentrating sequences, as illustrated by the broad range of applications presented below. Some examples of concentrations, together with their decompositions as above, can be found in the Appendix.

Moreover, the Vanishing Mass Conjecture may be considered a far-reaching generalization of the main theorem of~\cite{DPR16}, which for a first-order linear differential operator
\begin{equation*}
   \mathcal A := \sum_{k=1}^d A_k \partial_k
\end{equation*}
with constant coefficients $A_k\in{\rm Lin}(V,W)$, where $V$ and $W$ are finite-dimensional real vector spaces, reads as follows:

\begin{theorem}
\label{thm:singular_density}
Let $\mathcal{A}$ be a constant-coefficient linear differential operator and let $\mu\in\mathcal{M}(\Omega;V)$ be a finite Radon measure, where $\Omega\subseteq\R^d$ is an open set, with Lebesgue--Radon--Nikodym decomposition
\[
  \mu = \mu^a + \mu^s = \frac{\de\mu}{\de|\mu|} \, \bigl( |\mu|^a + |\mu|^s \bigr).
\]
Suppose that
\begin{equation*}
    \mathcal{A}\mu = 0  \qquad\text{or}\qquad
    \mathcal{A}\mu = \sigma \in\mathcal{M}(\Omega;W)
\end{equation*}
distributionally. Then,
\begin{equation}
\label{eq:sing_density}
   \frac{\de\mu}{\de|\mu|}(x)\in\Lambda_{\mathcal{A}}\quad\text{for $|\mu|^s$-a.e.\ $x$,}
\end{equation}
where
\[
   \Lambda_{\mathcal A} := \bigcup_{|\xi|=1} \ker\mathbb A(\xi),
    \qquad
    \mathbb A(\xi) := \sum_{k=1}^d A_k \xi_k,
    \quad
    \xi \in \R^d,
\]
is the \emph{wave cone} associated to $\mathcal{A}$.
\end{theorem}

The wave cone was introduced by Tartar in the theory of compensated compactness (see, e.g., Section~8.2 and the Notes to Chapter~8 in~\cite{R26}) and it contains all the vectors $v \in V$ such that $\mathcal{A}v$ is \emph{not} elliptic. For $\mathcal{A} = \mathrm{div}$ it is given in~\eqref{eq:Lambda_div}. For $\mathcal{A} = \rm{curl}$, Theorem~\ref{thm:singular_density} is precisely Alberti's celebrated Rank-One Theorem~\cite{A93,DL08,MV19}, which constrains the allowed singularities in the derivatives of BV-functions, as discussed, for instance, in Section~13.5 of~\cite{R26}.

Theorem~\ref{thm:singular_density} is limited to a statement concerning the singular part of the measure $\mu$. It says nothing about the absolutely continuous part of $\mu$ and, in fact, nothing can be said about it in general. For instance, the beautiful and simple result in~\cite[Theorem~3]{A91} (see also \cite{DMG25}) entails that \emph{any} integrable vector field is the absolutely continuous part of the derivative of some BV-function, so that there cannot be any constraints on it.

However, even if a concentrating sequence converges to a (Lebesgue) absolutely continuous measure, the Vanishing Mass Conjecture makes a non-trivial assertion, namely that the value distribution (Young) measure $\mu$ is in fact a superposition of \emph{special} probability measures (represented by the $\mu_\xi$), namely those with barycenter in the wave cone. In this way, the shape of the wave cone $\Lambda_\mathcal{A}$ is reflected in the structure of the asymptotic concentrations. This understanding of the Vanishing Mass Conjecture is reinforced by the fact that Theorem~\ref{thm:singular_density}, and with it all its applications, is in fact a straightforward corollary of the claim of the Vanishing Mass Conjecture (see Theorem~\ref{thm:giancarlo} below).

We remark that in the recent paper~\cite{BW25} the authors were able to slightly improve Theorem~\ref{thm:singular_density}. Indeed, even if not explicitly stated there, they proved that if $(u_j)_j$ is an $\mathcal{A}$-free sequence on a Lipschitz domain $\Omega$ converging to zero in measure and converging \emph{strictly} to some limit $u$, that is $u_j \toweakstar u$ and $\|u_j\|_{{\rm L}^1}\to\|u\|_{\rm TV}$ as $j\to\infty$ (where $\|\frarg\|_{\rm TV}$ denotes the total variation norm), then the corresponding concentration Young measure $\mu$ is entirely supported in the wave cone $\Lambda_{\mathcal{A}}$ (see their Theorem~1.9). However, the general picture, without strict convergence, is necessarily much more complex, cf.~Example~3 in the Appendix. Other related results were proved in~\cite{ARDPHRS24,ARA26}.

Our main theorems will assert the claim of Conjecture~\ref{conj:vmc} in full generality, for first-order constant-coefficient linear differential operators $\mathcal{A}$. To state them precisely, we let $Q:=(0,1)^d\subseteq\R^d$ be the open unit cube in $\R^d$ and we introduce the following definitions: For a function $X\in {\rm L}^1(Q;V)$, its \emph{distribution of directions} will be the measure
\begin{equation} \label{eq:Theta}
    \Theta_X:=\bigg(\frac{X}{|X|}\bigg)_\sharp(|X| \, \mathscr L^d \mres Q)\in\mathcal{M}(V),
\end{equation}
where in $|\frarg|$ we dropped the subscript $V$ and $X/|X|$ is defined arbitrarily on $\{|X|=0\}$. If $V$ is a space of vectors or matrices, then $|\frarg|$ always denotes the Euclidean or Frobenius norm, respectively. We also set $\mathbb S_V:=\{v\in V:\,|v|=1\}$.

\begin{definition}
\label{def:pure_concentrations}
    We say that a probability measure $\mu\in\mathcal{M}^1(\mathbb S_V)$ belongs to the set $\mathcal{Y}_\mathcal{A}(Q)$ of \emph{concentration Young measures}, or simply $\mathcal{Y}_\mathcal{A}$, if there exists a sequence $(u_j)_j \subseteq {\rm C}^\infty(Q;V)$ and a measure $u\in\mathcal{M}(\overline{Q};V)$ such that 
    \begin{enumerate}[(i)]
        \item $\mathcal{A}u_j=0$ in $Q$ (distributionally) for all $j$;
        \item $\|u_j\|_{{\rm L}^1(Q;V)}=1$ for all $j$;
        \item $u_j$ converges to zero in measure as $j \to \infty$, that is, $|\{|u_j| > \delta\}| \to 0$ for all $\delta > 0$;
        \item $u_j\toweakstar u$ in $\mathcal{M}(\overline{Q};V)$ as $j \to \infty$ (weak* convergence), with $|u|(\partial Q)=0$;
        \item $\Theta_{u_j}\toweakstar\mu$ in $\mathcal{M}^1(\mathbb S_V)$ as $j \to \infty$.
    \end{enumerate}
\end{definition}

\begin{remark}
\label{rem:cube_not_Omega}
We are exclusively working in the domain $Q=(0,1)^d$ for notational reasons, but any other domain $\Omega\subseteq\R^d$ with Lipschitz boundary is possible as well. Note moreover that a concentrating sequence as above cannot be equi-integrable.
\end{remark}
\begin{remark}
    We note that in~(i) above one could also allow for $\mathcal{A}u_j=w_j$, with $\|w_j\|_{{\rm L}^1}$ uniformly bounded and $w_j\to 0$ in measure. This generalization will be clear from the proofs.
\end{remark}
We also denote by $\overline{\rm co}^{w*}\, \mathcal{W}_{\Lambda_\mathcal{A}}$ the weak* closure of the convex hull of the set
\begin{equation}
    \mathcal{W}_{\Lambda_{\mathcal{A}}}:=\bigl\{\mu\in\mathcal{M}^1(\mathbb S_V) \;:\; [\mu]\in\Lambda_\mathcal{A}\bigr\}.
\end{equation}

\begin{theorem}
\label{thm:vanishing_mass}
Let $\mu\in\mathcal{Y}_\mathcal{A}$, with $\mathcal{A}$ a first-order constant-coefficient linear differential operator. Then, $\mu\in\overline{\rm co}^{w*}\, \mathcal{W}_{\Lambda_{\mathcal{A}}}$ and there exists a probability measure $\pi\in\mathcal{M}^1(\mathcal{W}_{\Lambda_{\mathcal{A}}})$ such that 
\begin{equation}
    \label{eq:vanishing_mass}
    \mu=\int_{\mathcal{W}_{\Lambda_{\mathcal{A}}}}\nu\de\pi(\nu).
\end{equation}
In particular, Conjecture~\ref{conj:vmc} holds true.
\end{theorem}

Theorem~\ref{thm:vanishing_mass} does not yet imply that the measures in the decomposition belong themselves to the class $\mathcal{Y}_\mathcal{A}$ and indeed this might not be the case in general. However we can, in the restricted class of \emph{constant-rank} operators, that is,
    \begin{equation}
        \label{eq:constant_rank}
        {\rm rank}\, \mathbb A(\xi)={\ell}  \quad\text{for all $\xi\neq 0$}
    \end{equation}
    with a constant $\ell \in \N$, and such that $\Lambda_{\mathcal{A}}$ is \emph{spanning}, that is,
    \begin{equation}
        \label{eq:spanning}
        {\rm span}\, \Lambda_\mathcal{A}=V,
    \end{equation}
infer that indeed the $\nu$'s in the decomposition~\eqref{eq:vanishing_mass} can be chosen to belong to $\mathcal{Y}_\mathcal{A}$:

\begin{theorem}
\label{thm:vanishing_mass2}
In the situation of Theorem~\ref{thm:vanishing_mass} assume additionally that $\mathcal{A}$ is of constant rank and $\Lambda_{\mathcal{A}}$ is spanning (this class includes the divergence and the curl). Then, in the decomposition~\eqref{eq:vanishing_mass} we may additionally require that $\nu\in\mathcal{Y}_\mathcal{A}$ for $\pi$-almost every $\nu$. Moreover, $\overline{\rm co}^{w*}\, \mathcal{W}_{\Lambda_{\mathcal{A}}} = \mathcal{Y}_\mathcal{A}$.
\end{theorem}

The proof of these results rests on three central ideas: First, in Proposition~\ref{prop:Young_convexity} we show the convexity of $\mathcal{Y}_{\mathrm{A}}$ without fixing the barycenter (the convexity of generalized Young measures is classical when the barycenter is held fixed).

Second, Propositions~\ref{prop:convex_dicomoty},~\ref{prop:Choquet_repr} together characterize -- in the spirit of Choquet's theorem -- the decomposition~\eqref{eq:vanishing_mass} via the non-negativity of all integrals of a ``barrier function'' $q$ with respect to the measure $\mu$ under investigation. These barrier functions are the convex and positively $1$-homogeneous functions that are non-negative on the wave cone $\Lambda_{\mathcal{A}}$.

Finally, we establish the non-negativity of said integrals via a compensated compactness argument, which is stated in Proposition~\ref{pro[:compensated_compactness_convex} (this is Theorem~\ref{thm:compensated_compactness} under an additional convexity assumption, but is established before all the other theorems). The intricate proof is based on identities and estimates for the heat flow, which replaces classical Fourier methods since the latter are not available in an ${\rm L}^1$-setting.

In this context let us recall that the classical Tartar theorem of compensated compactness (see~\cite{T79,T83}, reproduced in~\cite[Theorem 8.7]{R26}) only applies to quadratic forms, while our barrier functions are not necessarily square roots of quadratic forms, but merely positively $1$-homogeneous. Nevertheless, our proof exhibits some parallels to Tartar's classical argument: There, it is shown that for any quadratic form $Q$ that is non-negative on $\Lambda_\mathcal{A}$ and for every $\delta>0$ one may find a constant $C_\delta>0$ such that
\begin{equation*}
    Q(Z)\geq-\delta|Z|^2-C_\delta|\mathbb A(\xi)Z|^2,  \qquad Z \in \C^d,
\end{equation*}
see~\cite[Equation (8.11)]{R26}. This estimate roughly says that if $Q(Z)$ is negative, then $Z$ must be quantitatively away from the wave cone. In this regard, the core estimate~\eqref{eq:almost_exposed} in our proof is analogous, even though it is established in a very different way.

We also consider various applications of our main results, Theorems~\ref{thm:vanishing_mass} and~\ref{thm:vanishing_mass2}. The first one is a new proof of Theorem~\ref{thm:singular_density}:

\begin{theorem}
\label{thm:giancarlo}
Let $\mathcal{A}$ be a first-order constant-coefficient linear differential operator and let $\mu \in \mathcal{M}(\Omega;V)$ be an $\mathcal{A}$-free measure (i.e., $\mathcal{A}\mu = 0$). Then Theorem~\ref{thm:singular_density} holds true. 
\end{theorem}

As a second application, for first-order constant-rank differential operators, we remove the aperture-smallness assumption from the equi-integrability (compensated compactness) results of~\cite{ARDPHRS24,ARA26}, extending them to arbitrary convex—and hence one-sided—cones.

\begin{theorem}
\label{thm:cone}
Let $\mathcal{A}$ be as in Theorem~\ref{thm:vanishing_mass} with the additional assumption that it is of constant rank. Let $K\subseteq V$ be a closed convex cone such that $K\cap\Lambda_\mathcal{A}=\{0\}$ and $K\cap(-K)=\{0\}$. Let $\Omega\subseteq\R^d$ be an open Lipschitz domain and let $(\mu_j)_{j\in\N}\subseteq\mathcal{M}(\Omega;V)$ be such that 
\begin{equation}
    \label{eq:A_free_measures}
    \mathcal{A}\mu_j=0,\qquad\sup_{j\in\N}|\mu_j|(\Omega)<+\infty,
\end{equation}
and
\begin{equation}
    \label{eq:distance_from_K}
    \lim_{j\to\infty}\int_{\Omega}{\rm dist}\bigg(\frac{\de\mu_j}{\de|\mu_j|}(x),K\bigg)\de|\mu_j|(x)=0.
\end{equation}
Then, for every $\Omega^\prime\subseteq\Omega$, writing $\mu_j\mres\Omega^\prime=g_j\mathscr L^d\mres\Omega^\prime+\mu_j^s$ with $g_j \in {\rm L}^1(\Omega^\prime)$, it holds that the family $(g_j)_j$ is equi-integrable in ${\rm L}^1(\Omega^\prime;V)$ and $|\mu^s_j|(\Omega^\prime)\to 0$ as $j\to\infty$. Moreover, if even
\[
  \frac{\de\mu_j}{\de|\mu_j|}(x) \in K \qquad\text{for $|\mu_j|$-a.e.\ $x$ and for all $j \in \N$,}
\]
then $\mu_j^s = 0$ and there is $p = p(K) \in (1,\infty)$ (implicitly also depending on $\mathcal{A}$) such that for all $\Omega^\prime\subseteq\Omega$,
\[
  \sup_{j\in\N} \|\mu_j\|_{{\rm L}^p(\Omega^\prime)} < \infty.
\]
\end{theorem}

Next, we formulate a compensated compactness result in the spirit of Tartar's theorem (see~\cite[Theorem 8.7]{R26}), but in the linear growth case, not for quadratic forms. For this, recall that a function $f$ is called $\Lambda_\mathcal{A}$-convex if it is convex along any direction in $\Lambda_{\mathcal{A}}$, that is,
\begin{equation} \label{eq:Lambda_A_qc}
  f((1-\theta) v_0 + \theta v_1) \leq (1-\theta) f(v_0) + \theta f(v_1)
\end{equation}
for all $v_0, v_1 \in V$ with $v_1 - v_0 \in \Lambda_{\mathcal{A}}$ and all $\theta \in (0,1)$. It is well-known (see, e.g.,~\cite{FM99,KK16}) that the so-called $\mathcal{A}$-quasiconvexity implies $\Lambda_\mathcal{A}$-convexity. For a function $f:V\to\R$ we also define (whenever it exists) the \emph{strong recession function} $f^\infty:V\to\R$ as 
\begin{equation*}
    f^\infty(v)=\lim_{\underset{t\to+\infty}{v'\to v}}\frac{f(tv')}{t},  \qquad v\in V.
\end{equation*}

\begin{theorem} \label{thm:compensated_compactness}
Let $(u_j)_j$ be a sequence as in Definition~\ref{def:pure_concentrations} and let the assumptions of Theorem~\ref{thm:vanishing_mass2} concerning $\mathcal{A}$ be satisfied. Then, for every $f:V\to\R$ that is $\Lambda_{\mathcal{A}}$-convex (or convex), non-negative on $\Lambda_\mathcal{A}$, and for which $f^\infty$ exists, it holds that
\begin{equation*}
    \liminf_{j\to\infty}\int_Q f(A+u_j)\de x\geq f(A)\qquad\text{for all $A\in V$}.
\end{equation*}
\end{theorem}

One may put such a result into context with Morrey's Conjecture (see, e.g., Chapter~8 in~\cite{R26} for an overview on this central question in the Calculus of Variations) as follows: Since $\mathcal{A}$-quasiconvexity and $\Lambda_{\mathcal{A}}$-convexity do not agree in general (for the classical case $\mathcal{A} = {\rm curl}$ see~\cite{S92,FS08,AIPS12,G18,GT22,AFGKK24,C26}), lower semicontinuity for integral functionals requires $\mathcal{A}$-quasiconvexity, not merely $\Lambda_{\mathcal{A}}$-convexity (see, e.g.,~\cite{FM99}). For pure concentrations, however, the preceding result shows that $\Lambda_{\mathcal{A}}$-convexity suffices. So, in this sense one may say that concentrations are ``simpler'' than oscillations and that the analogue of Morrey's Conjecture is \emph{false} for pure concentrations, which we sum up as follows:

\begin{corollary}
\label{cor:Morrey_conj}
The validity of Morrey's Conjecture (that is, the non-equivalence of quasiconvexity and rank-one convexity) cannot be witnessed by a function with linear growth that is non-negative on the cone of rank-one matrices and for which $f^\infty$ exists, applied to a purely concentrating sequence (as in Definition~\ref{def:pure_concentrations}).
\end{corollary}

On an intuitive level, this may be explained as follows: The convergence to zero in measure in effect forces all individual concentration amplitudes present in a non-equi-integrable $\mathcal{A}$-free sequence to be $\Lambda_{\mathcal{A}}$-connected to the zero vector. This severely constrains the admissible ``shapes'' of these concentrations, ultimately collapsing their degrees of freedom so that only superpositions of simple concentrations remain possible.

Finally, for the following application we write ${\rm ex}\, \overline{\rm co}^{w*}\, \mathcal{W}_{\Lambda_{\mathcal{A}}}$ to denote the set of extremal points of $\overline{\rm co}^{w*}\, \mathcal{W}_{\Lambda_\mathcal{A}}$. It is well-known, and follows directly from the relevant definitions, that ${\rm ex}\, \overline{\rm co}^{w*}\, \mathcal{W}_{\Lambda_{\mathcal{A}}} \subseteq {\rm ex}\, \mathcal{W}_{\Lambda_{\mathcal{A}}}$. We also write $\#A$ for the \emph{cardinality} of a set $A$. 

\begin{theorem} \label{thm:extremals}
In Theorems~\ref{thm:vanishing_mass},~\ref{thm:vanishing_mass2}, the decomposition can be refined to
\[
  \mu=\int_{{\rm ex}\, \overline{\rm co}^{w*}\, \mathcal{W}_{\Lambda_{\mathcal{A}}}}\nu\de\pi(\nu),
\]
that is, $\pi\in\mathcal{M}^1({\rm ex}\, \overline{\rm co}^{w*}\, \mathcal{W}_{\Lambda_{\mathcal{A}}}) \subseteq \mathcal{M}^1({\rm ex}\, \mathcal{W}_{\Lambda_{\mathcal{A}}})$. Moreover, if $\nu\in{\rm ex}\, \mathcal{W}_{\Lambda_{\mathcal{A}}}$ with $[\nu]\in{\rm ker}\, \mathbb A(\xi) \subseteq \Lambda_{\mathcal{A}}$ for some $\xi\neq 0$, then ${\rm supp}\, \nu$ consists of isolated points and it holds that $\#{\rm supp}\, \nu\leq{\rm rank}\, \mathbb A(\xi)+1$. In particular, if $[\nu] = 0$, then $\#{\rm supp}\, \nu\leq \min_{\xi\neq 0}{\rm rank}\mathbb A(\xi)+1$.
\end{theorem}

We note that another application of the above results to the anisotropic Michael--Simon inequality has appeared very recently in~\cite{FT26}.

The paper is organized as follows: In Section~\ref{sec:convex_analysis} we will prove some general convexity results for concentration Young measures and in Section~\ref{sec:main_results} we introduce other tools and techniques, most notably various estimates on the heat flow. The main theorems and their applications are established in Section~\ref{sec:proofs}. Finally, in an Appendix we present some relevant examples of concentrations, which we hope will illustrate and motivate Conjecture~\ref{conj:vmc}. 

\begin{ackn}
This work was supported by UK Research and Innovation (UKRI) under the Horizon Europe funding guarantee [grant number EP/Z000297/1]. The authors would like to thank Giovanni Alberti, Adolfo Arroyo-Rabasa, Jean-Fran\c{o}is Babadjian, David Bate, Guido De~Philippis, Jan Kristensen, and Andrea Merlo for discussions related to the topic of this paper. The authors are particularly grateful to Guido De~Philippis for pointing out how to obtain higher integrability in Theorem~\ref{thm:cone}, and to Jan Kristensen for suggesting to remove a previous assumption of positive $1$-homogeneity in Theorem~\ref{thm:compensated_compactness} and Corollary~\ref{cor:Morrey_conj}, both during discussions at MFO. AI was used in the preparation of this paper to discuss a few minor ideas and estimates, but not in the writing of the paper besides some help with formulations.
\end{ackn}

\section{Convexity and barrier functions}
\label{sec:convex_analysis}

We start with the following lemma, which estimates Young measure representations for sums of concentrating sequences. Here, for a metric space $Z$ we denote by ${\rm Lip}(Z)$ the set of all Lipschitz functions $f : Z\to\R$. 

\begin{lemma}
\label{lem: almost additivity}
Let $Q\subseteq\R^d$ be open and let $\psi\in{\rm Lip}(\mathbb S_V)$. For any $W\in {\rm L}^1(Q;V)$ set
\begin{equation*}
   F_{\psi}(W):=\int_{\mathbb S_V}\psi\de\Theta_W,
\end{equation*}
where $\Theta_W$ is defined in~\eqref{eq:Theta}. Then, for all $U,V\in {\rm L}^1(Q;V)$,
\begin{equation}\label{eq:almost_additive}
  \bigl|F_{\psi}(U+V)-F_{\psi}(U)-F_{\psi}(V)\bigr|
  \leq C_\psi \int_{Q}\min\{|U(x)|,|V(x)|\}\de x,
\end{equation}
where $C_\psi :=6\|\psi\|_{\infty}+2\,\mathrm{Lip}(\psi)$.

\end{lemma}

\begin{proof}
We claim that for all $a,b\in V$ we have
\begin{equation} \label{eq:almost_additivity_claim}
\biggl||a+b|\psi\bigg(\frac{a+b}{|a+b|}\bigg)-|a|\psi\bigg(\frac{a}{|a|}\bigg)-|b|\psi\bigg(\frac{b}{|b|}\bigg)\biggr|
\leq C_\psi \,\min\{|a|,|b|\},
\end{equation}
where $C_\psi >0$ is as in the statement of the lemma. Applying this pointwise inequality with $a:=U(x)$ and $b:=V(x)$ before integrating over $Q$ gives~\eqref{eq:almost_additive}.

To prove the claim~\eqref{eq:almost_additivity_claim}, fix $a,b\in V$. If $a=0$ or $b=0$ then the claim is trivial; otherwise, assume $|a|\geq |b|$, so that $\min\{|a|,|b|\}=|b|$.

\medskip\noindent
\emph{Case 1: $|a|\geq 2|b|$.}\;
Then, $|a+b|\geq |a|-|b|\geq |a|/2$. Moreover,
\begin{equation*}
\left|\frac{a+b}{|a+b|}-\frac{a}{|a|}\right|
\leq \left|\frac{a+b}{|a+b|}-\frac{a+b}{|a|}\right|
    +\left|\frac{a+b}{|a|}-\frac{a}{|a|}\right|
=\frac{||a|-|a+b||}{|a|}+\frac{|b|}{|a|}
\leq \frac{2|b|}{|a|}.
\end{equation*}
Hence, with
\begin{equation*}
  \widehat{v}:=\frac{v}{|v|},  \qquad v \in V \setminus \{0\}, 
\end{equation*}
we have
\begin{equation*}
\bigl|\psi(\widehat{a+b})-\psi(\widehat a)\bigr|
\leq \mathrm{Lip}(\psi)\,\frac{2|b|}{|a|},
\end{equation*}
where $\mathrm{Lip}(\psi)$ denotes the Lipschitz constant of $\psi$. Using $||a+b|-|a||\leq |b|$ and the identity
\begin{equation*}
|a+b|\psi(\widehat{a+b})-|a|\psi(\widehat a)
=(|a+b|-|a|)\psi(\widehat{a+b})+|a|\bigl(\psi(\widehat{a+b})-\psi(\widehat a)\bigr),
\end{equation*}
it follows that
\begin{equation*}
\bigl||a+b|\psi(\widehat{a+b})-|a|\psi(\widehat a)\bigr|
\leq \|\psi\|_{\infty}|b|+2\,\mathrm{Lip}(\psi)\,|b|.
\end{equation*}
Therefore,
\begin{equation*}
\bigl||a+b|\psi(\widehat{a+b})-|a|\psi(\widehat a)-|b|\psi(\widehat b)\bigr|
\leq (2\|\psi\|_{\infty}+2\,\mathrm{Lip}(\psi))\,|b|.
\end{equation*}

\medskip\noindent
\emph{Case 2: $|a|<2|b|$.}\;
Then, $|a+b|\leq |a|+|b|<3|b|$, whereby
\begin{equation*}
\bigl||a+b|\psi(\widehat{a+b})-|a|\psi(\widehat a)-|b|\psi(\widehat b)\bigr|
\leq \|\psi\|_{\infty}\,(|a+b|+|a|+|b|)
\leq 6\|\psi\|_{\infty}|b|,
\end{equation*}
where if $a=-b$ we interpret $|a+b|\psi(\widehat{a+b})$ as zero.
Combining the two cases yields our claim~\eqref{eq:almost_additivity_claim}.
\end{proof}

As a consequence we then obtain the following convexity result:

\begin{proposition}
\label{prop:Young_convexity}
The set $\mathcal{Y}_\mathcal{A}$ is convex.
\end{proposition}
\begin{proof}
Let $\mu_1,\mu_2\in\mathcal{Y}_\mathcal{A}$ be generated in the sense of Definition~\ref{def:pure_concentrations} by two sequences $(u^1_j)_j$ and $(u^2_j)_j$, whose weak* limits we denote by $u^1$ and $u^2$, respectively. We need to prove that then also $t\mu_1+(1-t)\mu_2\in\mathcal{Y}_\mathcal{A}$.

Observe first that $\min\{|u_j^1|,|u_n^2|\}$ converges to zero in ${\rm L}^1(Q)$ as $n\to\infty$, for all fixed $j\in\N$. Indeed, $\min\{|u_j^1|,|u_n^2|\}\leq|u^1_j|$ and  we also have convergence to zero in measure. Hence, as $n\to\infty$, we can apply the dominated convergence theorem (or Vitali's convergence theorem) to see that there is a subsequence $n_j$ such that
\begin{equation}
\label{eq:overlap_to_zero}
    \int_{Q}\min\{|u^1_j|(x),|u^2_{n_j}|(x)\}\de x\leq\frac{1}{j},\qquad j\in\N.
\end{equation}
Set
\[
  w_j := \beta_{j,n_j} (tu^1_j+(1-t)u^2_{n_j})
\]
with $\beta_{j,n_j} > 0$ chosen so that $\|w_j\|_{{\rm L}^1}=1$ for all $j$. Clearly, $w_j\in {\rm C}^\infty(Q;V)$, $\mathcal{A}w_j=0$, $w_j \to 0$ in measure, and $w_j\toweakstar tu^1+(1-t)u^2$.

Now, given $\Psi\in{\rm Lip}(\mathbb S_V)$, we can apply~\eqref{eq:almost_additive}, which together with~\eqref{eq:overlap_to_zero} gives that $(w_j)_j$ is a generating sequence (in the sense of Definition~\ref{def:pure_concentrations}) for $t\mu_1+(1-t)\mu_2$. This proves the proposition because the weak* convergence of probability measures on a compact metric space can be characterized by testing only with functions in the Lipschitz class by density.
\end{proof}

Let $\Gamma$ be a closed (possibly non-convex) cone in $V$ that contains the origin. Our next result is a duality characterization of the weakly* closed convex hull of the set
\begin{equation*}
  \mathcal{W}_{\Gamma}:=\big\{\mu\in\mathcal{M}^1(\mathbb S_V) \;:\; [\mu]\in\Gamma\big\}
\end{equation*}
via Jensen-type inequalities for all the convex, positively $1$-homogeneous functions $q:V\to\R$ with the property that $q\geq 0$ on $\Gamma$. In the following, we refer to these $q$ as \emph{barrier functions} associated to the cone $\Gamma$.

\begin{proposition}
\label{prop:convex_dicomoty}
Let $\Gamma$ be a closed cone in $V$ that contains the origin. For all $\mu\in\mathcal{M}^1(\mathbb S_V)$,
\begin{equation}
    \label{eq:WGamma_equivalence}
\mu\in\overline{\rm co}^{w*}\, \mathcal{W}_{\Gamma}\iff\int_{\mathbb S_V}q\de\mu\geq 0 \text{ for all barrier functions associated to the cone $\Gamma$.}
\end{equation}
\end{proposition}

\begin{remark}
Note that while $\mu\in\overline{\rm co}^{w*}\, \mathcal{W}_{\Gamma}$ must be a probability measure, one cannot conclude that $[\mu] \in \Gamma$ since $\Gamma$ is not convex.
\end{remark}

\begin{proof}
\noindent\emph{``$\Longrightarrow$'':}
We first suppose that $\mu\in \mathcal{W}_{\Gamma}$. Let $q$ be a barrier function for $\Gamma$. By Jensen's inequality we get
\begin{equation*}
    \int_{\mathbb S_V}q\de\mu\geq q([\mu])\geq 0.
\end{equation*}
Taking convex combinations of measures in $\mathcal{W}_{\Gamma}$ and their weak* limits clearly preserves this inequality, so the first implication follows.

\medskip
\noindent\emph{``$\Longleftarrow$'':}
Assume by contradiction that the right-hand condition in~\eqref{eq:WGamma_equivalence} holds, but that $\mu\notin\overline{\rm co}^{w*}\, \mathcal{W}_{\Gamma}$. Since $\overline{\rm co}^{w*}\, \mathcal{W}_{\Gamma}$ is a convex, weakly* compact subset of $\mathcal{M}(\mathbb S_V)$, by the Hahn--Banach theorem (see~\cite[Theorem 3.4(b)]{Rudin91} for a version that is applicable here) we can find $h\in {\rm C}(\mathbb S_V)$ and $c\in\R$ such that 
\begin{equation*}
    \int_{\mathbb S_V}h\de\mu<c\leq\int_{\mathbb S_V}h\de\sigma  \qquad
    \text{for all $\sigma\in\overline{\rm co}^{w*}\, \mathcal{W}_{\Gamma}$.}
\end{equation*}
Replacing $h$ with $h-c\in {\rm C}(\mathbb S_V)$ gives
\begin{equation} \label{eq:HB_assumption}
    \int_{\mathbb S_V}h\de\mu<0\leq\int_{\mathbb S_V}h\de\sigma  \qquad
    \text{for all $\sigma\in\overline{\rm co}^{w*}\, \mathcal{W}_{\Gamma}$.}
\end{equation}
We shall refer to~\eqref{eq:HB_assumption} as the ``Hahn--Banach assumption''.

Now we define
\begin{equation*}
    q_h(x):=\inf\bigg\{\sum_{i=1}^Nt_ih(v_i) \;:\; N\in\N,\; t_i\geq0,\; v_i\in\mathbb S_V,\; \sum_{i=1}^Nt_iv_i=x\bigg\},  \qquad x\in V.
\end{equation*}
It is immediate to see that $q_h(x)<+\infty$ for all $x\neq 0$ by choosing $N:=1$, $t_1 := |x|$ and $v_1 := x/|x|$. We can also prove that $q_h$ is never equal to $-\infty$: Let $(t_i,v_i)_{i=1}^N$ be as in the definition of $q_h$ and set $a:=\sum_{i=1}^Nt_ih(v_i)$. By appending the pair $(t_{N+1},v_{N+1}):=(|x|,-x/|x|)$ one observes that $\sum_{i=1}^{N+1}t_iv_i=0$ and hence
\[
  \sum_{i=1}^{N+1}t_ih(v_i) = a+|x|h(-x/|x|).
\]
Moreover, the measure
\begin{equation*}
    \nu:=\frac{\sum_{i=1}^{N+1}t_i\delta_{v_i}}{\sum_{i=1}^{N+1}t_i}\in\mathcal{M}^1(\mathbb S_V)
\end{equation*}
belongs to $\mathcal{W}_{\Gamma}$ since $[\nu]=0$. Therefore, thanks to the Hahn--Banach assumption~\eqref{eq:HB_assumption},
\begin{equation*}
    0\leq\int_{\mathbb S_V}h\de\nu = \frac{a+|x|h(-x/|x|)}{\sum_{i=1}^{N+1}t_i},
\end{equation*}
whence $a\geq-|x|h(-x/|x|)$. Minimizing over all possible choices, the latter implies
\[
  q_h(x)\geq-|x|h(-x/|x|) > -\infty.
\]
Furthermore, the convexity and positive one-homogeneity of $q_h$ can be easily checked to hold.

Evidently, $q_h \leq h$, first on $\mathbb S_V$ and then on all of $V$ by positive $1$-homogeneity. Conversely, if $x\in\Gamma$ then $q_h(x)\geq 0$ thanks to the Hahn--Banach assumption. Indeed, consider any decomposition $(t_i,v_i)_{i=1}^N$ with $\sum_{i=1}^Nt_iv_i=x$ and set $T:=\sum_{i=1}^Nt_i$. For the probability measure
\begin{equation*}
    \sigma:=\frac{1}{T}\sum_{i=1}^Nt_i\delta_{v_i}\in\mathcal{M}^1(\mathbb S_V)
\end{equation*}
we have that $[\sigma]=x/T \in\Gamma$ (since $\Gamma$ is a cone). Via the Hahn--Banach assumption~\eqref{eq:HB_assumption},
\begin{equation*}
    0\leq T\int_{\mathbb S_V}h\de\sigma=\sum_{i=1}^Nt_ih(v_i),
\end{equation*}
whence $q_h(x)\geq 0$ follows. In conclusion, we have shown that $q_h$ is a barrier function associated to the cone $\Gamma$.

Therefore, applying the right-hand condition in~\eqref{eq:WGamma_equivalence} together with the strict inequality in the Hahn--Banach assumption~\eqref{eq:HB_assumption}, we get
\begin{equation*}
    0\leq \int_{\mathbb S_V}q_h\de\mu\leq\int_{\mathbb S_V}h\de\mu<0,
\end{equation*}
which is a contradiction.
\end{proof}

The following is a decomposition result which follows from Choquet's theorem (see, e.g.,~\cite{P01}), giving a representation of the measures belonging to $\overline{\rm co}^{w*}\, \mathcal{W}_{\Gamma}$. Since its proof is relatively short, we include it for the sake of exposition.

\begin{proposition}
\label{prop:Choquet_repr}
Let $\Gamma$ be a closed cone in $V$ that contains the origin. Then, the set $\mathcal{W}_{\Gamma}$ is weakly* compact and for every $\mu\in\overline{\rm co}^{w*}\, \mathcal{W}_{\Gamma}$ there exists a probability measure $\pi\in\mathcal{M}^1(\mathcal{W}_{\Gamma})$ such that 
\begin{equation}
    \label{eq:Choquet_repr}
    \mu = \int_{\mathcal{W}_{\Gamma}}\sigma\de\pi(\sigma).
\end{equation}
\end{proposition}

\begin{proof}
The barycenter map $\sigma\mapsto[\sigma]$ is continuous on $\mathcal{M}^1(\mathbb S_V)$ with respect to the weak* convergence and $\Gamma$ is closed. Hence, $\mathcal{W}_{\Gamma}$ is a weakly* closed subset of $\mathcal{M}^1(\mathbb S_V)$ and thus it is itself weakly* compact. Consequently, also $\mathcal{M}^1(\mathcal{W}_{\Gamma})$ is weakly* closed, hence weakly* compact. Moreover, the map
\begin{equation*}
    \pi\in\mathcal{M}^1(\mathcal{W}_{\Gamma})\mapsto [\pi] = \int_{\mathcal{W}_{\Gamma}}\sigma\de\pi(\sigma)\in\mathcal{M}^1(\mathbb S_V)
\end{equation*}
is continuous with respect to the weak* topologies of the domain $\mathcal{M}^1(\mathcal{W}_{\Gamma})$ and of the target $\mathcal{M}^1(\mathbb S_V)$. This implies that its image is weakly* compact. Also, since $\mathcal{M}^1(\mathcal{W}_{\Gamma})$ is convex and $\pi \mapsto [\pi]$ is linear, the image of our map is convex. In particular, the image contains the set of all finite convex combinations of points in $\mathcal{W}_{\Gamma}$ and then also their weak* closure. This proves that $\overline{\rm co}^{w*}\, \mathcal{W}_{\Gamma}$ lies in the image of the map $\pi \mapsto [\pi]$, implying the claim.
\end{proof}

Finally, we have the following simple technical result, which is required since we are not in general able to prove that the set $\mathcal{Y}_\mathcal{A}$ is weakly* compact (due to potential boundary concentrations).

\begin{lemma}
\label{lem:hom_weak_star_closed}
    Let $\mathcal{Y}_\mathcal{A}^{\rm unif}\subseteq\mathcal{Y}_\mathcal{A}$ be the set of measures in Definition~\ref{def:pure_concentrations} whose generating sequences $(u_j)_j$ additionally satisfy $|u_j|\toweakstar\mathscr L^d \mres Q$. Then, $\mathcal{Y}_\mathcal{A}^{\rm unif}$ is convex and weakly* closed.
\end{lemma}
\begin{proof}
The convexity of $\mathcal{Y}_\mathcal{A}^{\rm unif}$ follows as in the proof of Proposition~\ref{prop:Young_convexity}. For the weak* closedness, let $\mu$ be in the weak* closure of $\mathcal{Y}_\mathcal{A}^{\rm unif}$. Then, there exists a sequence $(\mu_n)_n\subseteq \mathcal{Y}_\mathcal{A}^{\rm unif}$ such that $\mu_n\toweakstar\mu$. Indeed, for the dual of a separable Banach space, the weak* topology on norm-bounded sets is metrizable, see, e.g.,~\cite[Theorem 3.16]{Rudin91}; concretely, one could use the $1$-Wasserstein distance on $\mathbb S_V$. We also recall that the convergence in measure is metrizable via $\rho(f,g) := \int_{\mathbb S_V} \min\{|f-g|,1\} \de H$ for measurable $f,g : \mathbb S_V \to \R$ (where $H$ is the canonical measure on $\mathbb S_V$). Consequently, by a diagonalization procedure, we can find a sequence $(u_n)_n \subseteq {\rm C}^\infty(Q;V)$ such that $(u_n)_n$ generates $\mu$ in the sense of Definition~\ref{def:pure_concentrations}. In this context note that condition~(iv) is preserved due to our additional assumption $|u_j|\toweakstar\mathscr L^d \mres Q$, which prevents concentration at the boundary. Therefore, $\mu\in\mathcal{Y}_\mathcal{A}^{\rm unif}$, proving the lemma.
\end{proof}

\section{Compensated compactness for concentrations}
\label{sec:main_results}

The goal of this section is to establish the following compensated compactness result, which is the same statement as Theorem~\ref{thm:compensated_compactness} except that the integrand is assumed to be \emph{convex}, not just $\Lambda_{\mathcal{A}}$-convex.

\begin{proposition} \label{pro[:compensated_compactness_convex}
Let $(u_j)_j$ be a sequence as in Definition~\ref{def:pure_concentrations}. Then, for every $q:V\to\R$ that is convex, positively $1$-homogeneous, and non-negative on $\Lambda_\mathcal{A}$, it holds that
\begin{equation*}
    \liminf_{j\to\infty}\int_Qq(u_j)\de x\geq 0.
\end{equation*}
\end{proposition}

An important technical tool we will employ in the proof is the following: For $f\in {\rm L}^1(\R^d)$ (either scalar or vector-valued) denote the \emph{heat flow} of $f$ at time $t>0$ as
\begin{equation}
    \label{eq:heat_flow}
    \mathcal{P}_tf(x):=\frac{1}{(4\pi t)^{d/2}}\int_{\R^d}{\rm e}^{-|x-y|^2/(4t)}f(y)\de y,  \qquad x \in \R^d,
\end{equation}
that is, the convolution of $f$ with the \emph{heat kernel}
\begin{equation}
    \label{eq:heat_kernel}
    p_t(z):=\frac{1}{(4\pi t)^{d/2}}{\rm e}^{-|z|^2/(4t)},  \qquad z \in \R^d.
\end{equation}
It is well-known that
\[
  \partial_t \mathcal{P}_tf=\Delta \mathcal{P}_t f
\]
for all $t>0$.

In all of the following (until the proof of Proposition~\ref{pro[:compensated_compactness_convex}) fix $s\in(0,1/2)$ and define
\[
  D := Q_s := (s,1-s)^d.
\]
Moreover, for $t>0$ set
\[
  a_t:=\mathcal{P}_t\ONE_D \in {\rm C}^\infty(\R^d).
\]

Let $v \in {\rm C}^\infty_c(\R^d;V)$ be a non-trivial (not everywhere zero), smooth, compactly supported vector field. With a henceforth fixed $T\in(0,1]$, we define
\begin{equation}
    \label{eq:parab_denisty}
    f:=|v|+\sqrt{T}|\mathcal{A}v|
\end{equation}
and
\[
    S:=\mathcal{P}_{T-\frarg}f,\quad V:=\mathcal{P}_{T-\frarg}v,\quad B:=\frac{V}{S},\quad\zeta:=\frac{\nabla S}{S}=\nabla\log S,\quad \eta:=\frac{\mathcal{P}_{T-\frarg}\mathcal{A}v}{S},
\]
together with the abbreviations
\[
    S_t:=\mathcal{P}_{T-t}f,\quad V_t:=\mathcal{P}_{T-t}v,\quad B_t:=\frac{V_t}{S_t},\quad\zeta_t:=\frac{\nabla S_t}{S_t}=\nabla\log S_t,\quad \eta_t:=\frac{\mathcal{P}_{T-t}\mathcal{A}v}{S_t}
\]
for $t \in [0,T)$ (as functions of $x$). The infinite speed of propagation for the solutions to the heat equation yields that $S > 0$ for $t > 0$ and hence all these quantities are indeed well-defined. For this and several other properties of the heat flow used in the following see~\cite{G09}. Furthermore, thanks to the order-preserving property of the heat semigroup,
\[
  |B_t|\leq 1.
\]
Indeed, we have 
\begin{equation*}
    |V(t,x)|\leq \mathcal{P}_{T-t}|v|(x)\leq \mathcal{P}_{T-t}f(x)=S(t,x),
\end{equation*}
where we used the triangle inequality and the monotonicity of the integral together with the fact that $|v|\leq f$.

In the following we denote by
\[
  \mathcal{H}:=\partial_t+\Delta
\]
the \emph{backward heat operator} and by
\[
  \mathcal{H}_\zeta := \mathcal{H}+2\zeta\cdot\nabla
\]
the \emph{backward heat operator with drift} (advection--diffusion operator). Clearly,
\[
  \mathcal{H}S=0  \qquad\text{and}\qquad
  \mathcal{H}V=0.
\]
Some further properties of the quantities defined above are contained in the following technical lemma.

\begin{lemma}
It holds that
\begin{equation}
    \label{eq:heat_identities}
    \mathcal{H}_\zeta B=0,\qquad \mathcal{H}_\zeta|B|^2=2|\nabla B|^2,\qquad \mathcal{H}_\zeta[\log S]=|\zeta|^2.
\end{equation}
Moreover,
\begin{equation}
    \label{eq:heat_estimate1}
    \int_0^T\int_{\R^d}aS|\nabla B|^2\de x\de t\leq\frac{1}{2}\int_{\R^d}a_Tf\de x
\end{equation}
and
\begin{equation}
  \label{eq:heat_estimate2}
    \int_0^T\int_{\R^d}aS|\eta|^2\de x\de t\leq\int_{\R^d}a_Tf\de x.
\end{equation}
\end{lemma}
\begin{proof}
Let $g$ be a smooth function on $[0,\infty)\times\R^d$ (either scalar-valued or vector-valued). We have the key identity
\begin{equation}
\label{eq:miracle1}
  \mathcal{H}(Sg)=S\partial_t g + S\Delta g+2\nabla S \cdot \nabla g = S \mathcal{H}_\zeta g,
\end{equation}
where $\nabla S \cdot \nabla g := \nabla S (\nabla g)^T$ (the scalar product between the row-vector $\nabla S$ and the row-vector(s) in $\nabla g$, which is just one row if $g$ is scalar-valued).

Thanks to~\eqref{eq:miracle1} applied with $g=B$,
\begin{equation*}
    0=\mathcal{H}V=\mathcal{H}(SB)=S\mathcal{H}_\zeta B,
\end{equation*}
which gives
\[
  \mathcal{H}_\zeta B=0,
\]
since $S>0$ everywhere. This is the first claim in~\eqref{eq:heat_identities}. Then also
\begin{equation*}
   \mathcal{H}_\zeta|B|^2=2B\partial_tB+2B\Delta B+2|\nabla B|^2+4B\zeta \cdot \nabla B
   = 2B \mathcal{H}_\zeta B+2|\nabla B|^2
   = 2|\nabla B|^2,
\end{equation*}
which is the second claim in~\eqref{eq:heat_identities}. To see also the third claim in~\eqref{eq:heat_identities}, we compute
\begin{equation*}
    \mathcal{H}_\zeta[\log S]=\frac{\mathcal{H}S}{S}-|\zeta|^2+2|\zeta|^2=|\zeta|^2.
\end{equation*}
   
Setting $F:=S|B|^2$ and using again~\eqref{eq:miracle1}, one obtains
\begin{equation*}
    \partial_t F + \Delta F = \mathcal{H}F=S\mathcal{H}_\zeta|B|^2=2S|\nabla B|^2.
\end{equation*}
 We can then compute, as $\partial_t a_t = \Delta a_t$ and for $t>0$,
    \begin{align*}
        \frac{\de}{\de t}\int_{\R^d}a_tF_t\de x
        &=\int_{\R^d}\big(\Delta a_t F_t+a_t\partial_tF_t\big)\de x \\
        &=\int_{\R^d}\big(\Delta a_t F_t-a_t\Delta F_t+2a_tS_t|\nabla B_t|^2\big)\de x \\
        &=2\int_{\R^d}a_tS_t|\nabla B_t|^2\de x,
    \end{align*}
    where for the last equality we integrated by parts the second term and canceled. We now integrate the previous identity in time from $0$ to $T-\delta$ to get 
\begin{align*}
    2\int_0^{T-\delta}\int_{\R^d}aS|\nabla B|^2\de x\de t&=\int_{\R^d}a_{T-\delta}S_{T-\delta}|B_{T-\delta}|^2\de x-\int_{D}S_0|B_0|^2\de x \\
    &\leq\int_{\R^d}a_{T-\delta}S_{T-\delta}\de x,
\end{align*}
where we used that $|B|\leq 1$ and neglected the second integral. Passing to the limit as $\delta\searrow 0$ yields~\eqref{eq:heat_estimate1}.

Finally,
\begin{equation*}
    S_t|\eta_t|^2=\frac{|\mathcal{P}_{T-t}\mathcal{A}v|^2}{S_t}\leq\frac{\mathcal{P}_{T-t}|\mathcal{A}v|}{\sqrt{T}},
\end{equation*}
where we used the triangle inequality to infer $|\mathcal{P}_{T-t}\mathcal{A}v|\leq \mathcal{P}_{T-t}|\mathcal{A}v|$ as well as the estimate $\sqrt{T}\mathcal{P}_{T-t}|\mathcal{A}v|\leq \mathcal{P}_{T-t}f = S_t$ via the pointwise bound $\sqrt{T}|\mathcal{A}v|\leq f$. Multiplying by $a$ and integrating over $(0,T)\times\R^d$,
\begin{equation*}
    \int_0^T\int_{\R^d}a_tS_t|\eta_t|^2\de x\de t\leq\frac{1}{\sqrt{T}}\int_0^T\int_{\R^d}a_t\mathcal{P}_{T-t}|\mathcal{A}v|\de x\de t.
\end{equation*}
We now use that the heat flow induces a semigroup of self-adjoint operators, i.e.,
\begin{equation} \label{eq:heat_semigroup}
    \int_{\R^d}\mathcal{P}_t f\mathcal{P}_s g\de x=\int_{\R^d}\mathcal{P}_{t+s} fg\de x,\qquad f,g\in {\rm L}^2(\R^d),
\end{equation}
to arrive at
\begin{equation*}
  \frac{1}{\sqrt{T}}\int_0^T\int_{\R^d}a_t\mathcal{P}_{T-t}|\mathcal{A}v|\de x\de t
  =\sqrt{T}\int_{\R^d}a_T|\mathcal{A}v|\de x\leq\int_{\R^d}a_Tf\de x.
\end{equation*}
Combining the last two estimates yields~\eqref{eq:heat_estimate2}.
\end{proof}

We now exploit the fact that the heat flow commutes with the differential operator $\mathcal{A}$ in order to derive a quantitative higher integrability estimate where $B$ is away from the wave cone. First observe that 
\begin{equation*}
    \mathcal{P}_{T-t}\mathcal{A}v=\mathcal{A}V_t=\mathcal{A}(S_tB_t).
\end{equation*}
So, using the Leibniz rule,
\begin{equation*}
    \mathcal{P}_{T-t}\mathcal{A}v
    =S_t\sum_{k=1}^d A_k \, \partial_kB_t + \sum_{k=1}^d\partial_kS_t\, A_kB_t
    =S_t\mathcal{A}B_t+\mathbb A(\nabla S_t)B_t.
\end{equation*}
We note in passing that this argument uses heavily that $\mathcal{A}$ is a first-order differential operator; for higher-order operators various commutators appear. Dividing by $S$,
\begin{equation}
    \label{eq:miracle2}
    \mathcal{A}B+\mathbb A(\zeta)B=\eta.
\end{equation}
This identity will allow us to prove the following key ellipticity lemma.

\begin{lemma}
Let $q:V\to\R$ be convex, positively $1$-homogeneous, and non-negative on $\Lambda_\mathcal{A}$. Then, for every $\eps>0$ there exists a constant $C=C(\mathcal{A},\eps,q)>0$ such that 
\begin{equation}
    \label{eq:higher_integr1}
    |\zeta|^2\leq C(|\nabla B|^2+|\eta|^2) \quad\text{in}\quad \{q(B)<-\eps\},
\end{equation}
and
\begin{equation}
    \label{eq:higher_integr2}
    \int_0^T\int_{\{q(B_t)<-\eps\}}aS|\zeta|^2\de x\de t\leq C\int_{\R^d}a_Tf\de x.
\end{equation}
\end{lemma}
\begin{proof}
We start by considering the compact set
\begin{equation*}
   K_{q,\eps}:=\bigl\{(z,\xi)\in V\times\mathbb S^{d-1}:\;|z|\leq 1,\;q(z)\leq-\eps\bigr\}\subseteq V\times\mathbb S^{d-1}.
\end{equation*}
Assuming that $K_{q,\eps}$ is non-empty (otherwise there is nothing to prove since $\{q(B)<-\eps\}$ would be empty), we define 
\begin{equation*}
   c_{q,\eps}:=\min_{(z,\xi) \in K_{q,\eps}}|\mathbb A(\xi)z|>0.
\end{equation*}
The claimed positivity of $c=c_{q,\eps}$ follow from the fact that $(z,\xi)\mapsto|\mathbb A(\xi)z|$ is continuous and if its minimum over $K_{q,\eps}$, attained at some $(z_0,\xi_0) \in K_{q,\eps}$, was zero, then $z_0 \in {\rm ker}\mathbb A(\xi_0) \subseteq \Lambda_{\mathcal{A}}$ and hence $q(z_0)\geq 0$ by assumption, contradicting the definition of $K_{\eps,q}$.

Now we use~\eqref{eq:miracle2} to estimate in $\{q(B)<-\eps\}$ as follows:
\begin{equation*}
   c_{q,\eps}|\zeta|\leq|\mathbb A(\zeta)B|\leq|\eta|+|\mathcal{A}B|\leq|\eta|+C_\mathcal{A}|\nabla B|.
\end{equation*}
Squaring both sides, applying Young's inequality, and rearranging gives~\eqref{eq:higher_integr1}. 

To prove~\eqref{eq:higher_integr2} we multiply both sides of~\eqref{eq:higher_integr1} by $aS$ and integrate to obtain
\begin{align*}
   \int_0^T\int_{\{q(B_t)<-\eps\}}aS|\zeta|^2\de x\de t&\leq C\int_0^T\int_{\{q(B_t)<-\eps\}}aS(|\nabla B|^2+|\eta|^2)\de x\de t \\
   &\leq C\int_{\R^d}a_Tf\de x,
\end{align*}
where in the last estimate we applied~\eqref{eq:heat_estimate1} and~\eqref{eq:heat_estimate2} (and the constant was doubled).
\end{proof}

The next result contains an a-priori estimate for the heat equation with a source term that is adapted to our situation.

\begin{lemma}
\label{lem:heat_reaction}
Let $0<\tau<T$ and suppose that $k:(0,T)\times\R^d\to[0,\infty)$ is a measurable function which satisfies
 \begin{equation}
 \label{eq:weight_integrability}
     \int_0^\tau\int_{\R^d}kaS\de x\de t<+\infty.
 \end{equation}
 Then there exists a (distributional) solution $h:(0,\tau)\times\R^d\to[0,\infty]$ of 
 \begin{equation}
     \label{eq:forcing_heat}
    \begin{cases}
     \partial_th=\Delta h-kh\qquad&{\rm in}\;(0,\tau)\times\R^d, \\
     h_0=\ONE_D&{\rm on}\;\R^d,
    \end{cases}
 \end{equation}
which satisfies
\[
  0\leq h\leq a.
\]
Moreover, setting
\[
  \rho:=Sh, \qquad \rho_t := S_t h_t,
\]
for all bounded test functions $\varphi\in {\rm C}^{1,2}([0,\tau]\times\R^d)$ such that
\begin{equation}
\label{eq:gaussian_bound}
|S\varphi|+|\nabla(S\varphi)|+|\Delta(S\varphi)|+|\partial_t(S\varphi)|\leq C{\rm e}^{-c|x|^2}, \qquad [0,\tau]\times\R^d,
\end{equation}
with a constant $C=C(\varphi)>0$, it holds that
\begin{equation}
    \label{eq:giga_heat}
    \int_{\R^d} \varphi_\tau\rho_\tau-\varphi_0\rho_0\de x +\int_0^\tau\int_{\R^d}k\varphi \rho\de x\de t=\int_0^\tau\int_{\R^d}\rho \mathcal{H}_\zeta\varphi\de x\de t,
\end{equation}
where, as usual, we have set $\varphi_t := \varphi(t,\frarg)$.
\end{lemma}
\begin{proof}
We proceed by approximation and let, for $n\in\N$, $k^n:=\min\{k,n\}$. A solution for the source density $k^n$ exists and is unique by~\cite[Proposition~2.3~(a)]{RS99}. Thanks to~\cite{B77} this solution coincides with the one provided by the Duhamel formula, namely
\begin{equation}
\label{eq:distr_Duhamel}
    h^n_t=a_t-\int_0^t\mathcal{P}_{t-s}(k^n_s h^n_s)\de s,\qquad t\in(0,\tau),
\end{equation}
where we recall that $a_t=\mathcal{P}_t\ONE_D$. By classical order estimates, see~\cite[Proposition 2.3]{RS99}, we also get
\[
  0\leq h^n\leq a  \qquad\text{and}\qquad
  h^{n+1}\leq h^n.
\]
In particular, the pointwise limit $h := \lim_{n\to\infty}h^n$ exists.

For any  $\phi \in {\rm C}^\infty_c((0,\tau) \times \R^d)$ there exists a constant $C_{\phi,\tau}>0$ such that 
\begin{equation*}
    |\phi|\leq C_{\phi,\tau}S
\end{equation*}
and consequently 
\begin{equation*}
    |\mathcal{P}_{t-s}\phi|\leq C_{\phi,\tau}S_s,  \qquad s \in (0,\tau).
\end{equation*}
The latter implies that $k^nh^n\mathcal{P}_{t-\frarg}\phi\leq C_{\phi,\tau}kaS\in {\rm L}^1([0,\tau]\times\R^d)$, with the integrability due to~\eqref{eq:weight_integrability}. Therefore, we can multiply~\eqref{eq:distr_Duhamel} by $\phi$, integrate, and apply the dominated convergence theorem to pass to the limit as $n \to \infty$, obtaining
\[
  \int_0^\tau \int_{\R^d} h_t\phi \de x \de t
  = \int_0^\tau \int_{\R^d} a_t -\int_0^t k_s h_s \mathcal{P}_{t-s}\phi \de s \de x \de t.
\]
Since this holds for all such $\phi$, we conclude that $h$ is a solution to~\eqref{eq:forcing_heat}. 

To prove~\eqref{eq:giga_heat} we start again from the PDE solved by $h$. Consider the distributional formulation of \eqref{eq:forcing_heat}, that is,
\begin{equation*}
 \int_{\R^d}\phi_\tau h_\tau-\phi_0h_0\de x-\int_0^\tau\int_{\R^d}h\partial_t\phi\de x\de t=\int_0^\tau\int_{\R^d}(h\Delta\phi-kh\phi)\de x\de t,
\end{equation*}
where $\phi$ is any test function in ${\rm C}^{1,2}([0,\tau]\times\R^d)$. Indeed, we may take $\phi$ from this class thanks to a standard cutoff procedure and the bounds \eqref{eq:gaussian_bound}. We can then choose $\phi:=S\varphi$ and rearrange the terms, so that
\begin{equation*}
     \int_{\R^d}\varphi_\tau\rho_\tau-\varphi_0\rho_0\de x=\int_0^\tau\int_{\R^d} h\mathcal{H}(S\varphi)-k\varphi\rho \de x\de t.
    \end{equation*}
Using that $\mathcal{H}(S\varphi)=S\mathcal{H}_\zeta\varphi$, we end up with \eqref{eq:giga_heat}.
\end{proof}

The next lemma contains our central quantitative compensated compactness estimate.

\begin{proposition}
\label{prop:almost_exposed}
Let $q:V\to\R$ be convex, positively $1$-homogeneous, and non-negative on $\Lambda_\mathcal{A}$. Set
\[
    M_{T}:=\int_{\R^d}a_Tf\de x,\qquad 
    M_{T,R}:=\int_{\{f\leq R\}}a_Tf\de x, \qquad
  N_q:=\max_{|z|\leq 1}\{-q(z),0\}.
\]
Then, for all $R_0,R>0$ such that  $\|S_0\|_{{\rm L}^\infty}=\|\mathcal{P}_Tf\|_{{\rm L}^\infty}<R_0<{\rm e}^{-1} R$, and for all $\eps>0$, there exists a constant $C=C(\mathcal{A},q,\eps)>0$ such that 
\begin{equation}
    \label{eq:almost_exposed}
    \int_{\R^d}a_Tq(v)\de x 
    \geq -\eps M_T-N_q\bigg(\frac{CM_T}{\log(R/R_0)}+M_{T,R}\bigg).
\end{equation}
\end{proposition}

\begin{proof}
Define
\begin{equation*}
 m^q(t,x):=\frac{\mathcal{P}_{T-t}(q(v))(x)}{S(t,x)}, \qquad
 m^q_t(x) := m^q(t,x).
\end{equation*}
Since $q$ is positively $1$-homogeneous, $|q(v)|\leq C_q|v|\leq C_q f$. Recall that $S_t=\mathcal{P}_{T-t}f$, so that $|m^q|\leq C_q$. We observe, once again by~\eqref{eq:miracle1}, that
\[
  S\mathcal{H}_\zeta m^q = \mathcal{H}(Sm^q)=\mathcal{H}(\mathcal{P}_{T-\frarg}(q(v)))=0.
\]
and hence, since $S>0$ (thanks to infinite speed of propagation), $\mathcal{H}_\zeta m^q=0$.

We now apply Jensen's inequality and the homogeneity of $q$ to obtain
\begin{equation}
\label{eq:m^q_bound}
 m^q\geq\frac{q(\mathcal{P}_{T-\frarg}v)}{S}=q(B).
\end{equation}
As $|B|\leq 1$, we conclude
\[
  m^q\geq-N_q,
\]

Define the measurable function
\[
   k(t,x):=|\zeta(t,x)|^2\ONE_{G_{\eps}}(t,x),
\]
where
\[
   G_{\eps}:=\{q(B)\geq -\eps\}\subseteq(0,T)\times\R^d,
\]
and we also define the time slice of $G_\eps$ as 
\begin{equation*}
    G_{\eps,t}:=\{q(B_t)\geq -\eps\}\subseteq\R^d.
\end{equation*}
We claim that this $k$ satisfies the integrability assumption~\eqref{eq:weight_integrability}. Indeed, we observe that thanks to the properties of the heat flow (and of the heat kernel),
\begin{equation*}
    \nabla S(t,x)=\int_{\R^d}\frac{y-x}{2(T-t)}p_{T-t}(x-y)f(y)\de y.
\end{equation*}
A direct application of Cauchy--Schwarz, writing the integrand as a product of the factors $\frac{y-x}{2(T-t)}\sqrt{p_{T-t}(x-y)f(y)}$ and $\sqrt{p_{T-t}(x-y)f(y)}$, gives 
\begin{equation*}
    |\nabla S(t,x)|^2 \leq S(t,x)\int_{\R^d}\frac{|y-x|^2}{4(T-t)^2}p_{T-t}(x-y)f(y)\de y.
\end{equation*}
For $0<\tau<T$ we then have, since $|a|\leq 1$,
\begin{align*}
\int_0^\tau\int_{\R^d}kaS\de x\de t&\leq\int_0^\tau\int_{\R^d}\frac{|\nabla S|^2}{S}\de x\de t \\
&\leq\int_0^\tau\int_{\R^d}f(y)\int_{\R^d}\frac{|y-x|^2}{4(T-t)^2}p_{T-t}(x-y)\de x\de y\de t \\
&\leq\|f\|_{{\rm L}^1}\int_0^\tau\int_{\R^d}\frac{|z|^2}{4(T-t)^2}p_{T-t}(z)\de z\de t\\
&\leq\frac{d}{2}\|f\|_{{\rm L}^1}\log\bigg(\frac{T}{T-\tau}\bigg)<+\infty,
\end{align*}
where we have used the standard fact that 
\begin{equation*}
    \int_{\R^d}|z|^2p_{T-t}(z)\de z=2d(T-t).
\end{equation*}
This shows our claim.

We can then apply Lemma~\ref{lem:heat_reaction} with the $k$ defined above, to obtain $h$ and $\rho=Sh$ solving~\eqref{eq:forcing_heat} and~\eqref{eq:giga_heat}, respectively. The test function $\varphi := \ONE$ is admissible since $S$ trivially satisfies all the required Gaussian bounds~\eqref{eq:gaussian_bound} ($\tau$ being away from $T$). So,~\eqref{eq:giga_heat} gives
\begin{equation*}
  M_T=E_\tau+K_\tau, \qquad{\rm where}\qquad
  E_\tau:=\int_{\R^d}\rho_\tau\de x,\quad
  K_\tau:=\int_{0}^\tau\int_{\R^d}k\rho\de x\de t.
\end{equation*}
Here we also used that via~\eqref{eq:heat_semigroup} it holds that
\[
  M_{T} = \int_{\R^d} a_T f \de x
  = \int_{\R^d} \ONE_D \mathcal{P}_T f\de x
  = \int_{\R^d} h_0 S_0 \de x
  = \int_{\R^d} \rho_0 \de x
\]
and that $\mathcal{H}_\zeta \ONE = 0$.

Likewise, the test function $\varphi:=m^q$ is also admissible in~\eqref{eq:giga_heat} since $Sm^q=\mathcal{P}_{T-\frarg}(q(v))$ satisfies the Gaussian bounds~\eqref{eq:gaussian_bound}. Using also that $\mathcal{H}_\zeta m^q = 0$ (see above), this yields
\begin{equation*}
    \int_{\R^d}a_Tq(v)\de x
    =\int_{\R^d}m^q_\tau\rho_\tau\de x+\int_0^\tau\int_{\R^d}km^q\rho\de x\de t.
\end{equation*}
Since $k=0$ on $G_\eps^c$,~\eqref{eq:m^q_bound} now gives the lower estimate
\begin{equation*}
    km^q\geq-k\eps,
\end{equation*}
and also
\begin{equation*}
   \int_{\R^d}m^q_\tau\rho_\tau\de x\geq\int_{\R^d}q(B_\tau)\rho_\tau\de x\geq-N_qE_\tau.
\end{equation*}
Combining all these estimates, we arrive at
\begin{equation}
\label{eq:almost_done1}
    \int_{\R^d}a_Tq(v)\de x
    \geq-N_qE_\tau-\eps K_\tau\geq-N_qE_\tau-\eps M_T,
\end{equation}
where we also used that $M_T=E_\tau+K_\tau\geq K_\tau$.

It remains to estimate $E_\tau$. For this, let $\eta\in {\rm C}^\infty(\R;[0,1])$ be non-decreasing and such that $\eta\equiv 0$ on $(-\infty,0]$ and $\eta\equiv 1$ on $[1,\infty)$. Set $J:=\log(R/R_0)>1$ and
\begin{equation*}
  \Phi(t,x):=\eta\bigg(\frac{\log\big(S(t,x)/R_0\big)}{J}\bigg).
\end{equation*}
Thanks to the properties of $\eta$ and to the fact that $|S_0|\leq\|\mathcal{P}_Tf\|_{{\rm L}^\infty}\leq R_0$ we get $\Phi(0,\frarg)=0$, while $\Phi\equiv 1$ on $\{S\geq R\}$. By~\eqref{eq:heat_identities} and careful computations one gets
\begin{equation*}
\mathcal{H}_\zeta\Phi= \eta^\prime\frac{\mathcal{H}S}{JS}-\eta^\prime\frac{|\nabla S|^2}{JS^2}+2\eta^\prime\frac{|\nabla S|^2}{JS^2}+\eta^{\prime\prime}\frac{|\nabla S|^2}{J^2S^2}=
0 + \bigg(\frac{\eta^\prime}{J}+\frac{\eta^{\prime\prime}}{J^2}\bigg)|\zeta|^2.
\end{equation*}
where we have abbreviated
\begin{equation*}
  \eta' := \eta^\prime \bigg(\frac{\log\big(S/R_0\big)}{J}\bigg),  \qquad
  \eta' := \eta^{\prime\prime}\bigg(\frac{\log\big(S/R_0\big)}{J}\bigg).
\end{equation*}
Thus,
\begin{equation*}
    |\mathcal{H}_\zeta\Phi|\leq\frac{C_\eta}{J}|\zeta|^2.
\end{equation*}

Since $S(t,x)\to 0$ as $|x|\to\infty$, uniformly in $[0,\tau]\times\R^d$, there exists a compact $K\subseteq\R^d$ so that $S(t,x)\leq R_0$ in $[0,\tau]\times K^c$. Hence, $\Phi$ has compact support inside $[0,\tau]\times\R^d$ and it is therefore an admissible test function in~\eqref{eq:giga_heat}, yielding
\begin{equation*}
    \int_{\R^d}\Phi_\tau\rho_\tau\de x+\int_0^\tau\int_{\R^d}k\Phi \rho\de x\de t=\int_0^\tau\int_{\R^d}\rho \mathcal{H}_\zeta\Phi\de x\de t.
\end{equation*}
We therefore get, discarding terms and using the properties of $\Phi$,
\begin{equation*}
    \int_{\{S_\tau\geq R\}}\rho_\tau\de x\leq \int_{\R^d}\Phi_\tau\rho_\tau\de x\leq\int_0^\tau\int_{\R^d}\rho \mathcal{H}_\zeta\Phi\de x\de t\leq\frac{C_\eta}{J}\int_0^\tau\int_{\R^d}\rho|\zeta|^2\de x\de t.
\end{equation*}
On the set $G_{\eps,t}^c=\{q(B_t)<-\eps\}$ we can use~\eqref{eq:higher_integr2}, along with $\rho=hS\leq aS$ (which is a consequence of the fact that $0\leq h\leq a$, see the proof of Lemma \ref{lem:heat_reaction}), to get
\begin{equation*}
  \frac{C_\eta}{J}\int_0^\tau\int_{G_{\eps,t}^c}\rho|\zeta|^2\de x\de t\leq\frac{C_\eta}{J}\int_0^\tau\int_{G_{\eps,t}^c}aS|\zeta|^2\de x\de t\leq \frac{C_{\mathcal{A},q,\eps,\eta}}{J}M_T,
\end{equation*}
where we absorbed constants. On the set $G_{\eps,t}$ instead we have
\begin{equation*}
    \frac{C_\eta}{J}\int_0^\tau\int_{G_{\eps,t}}\rho|\zeta|^2\de x\de t=\frac{C_\eta}{J}\int_0^\tau\int_{\R^d}k\rho\de x\de t\leq\frac{C_\eta}{J} K_\tau\leq\frac{C_\eta}{J} M_T.
\end{equation*}
Combining, we have
\begin{equation*}
    \int_{\{S_\tau\geq R\}}\rho_\tau\de x\leq C\frac{M_T}{J}.
\end{equation*}

For the integral on ${\{S_\tau<R\}}$ we instead get, using $\rho \leq aS$ again,
\begin{equation*}
    \int_{\{S_\tau<R\}}\rho_\tau\de x\leq\int_{\{S_\tau<R\}}a_\tau S_\tau\de x. 
\end{equation*}
Observe next that $S_\tau\to f$ uniformly and in ${\rm L}^1$ as $\tau\nearrow T$ since $f$ is Lipschitz and compactly supported. As a consequence, we also have $a_\tau S_\tau\to a_Tf$ in ${\rm L}^1$ as $\tau\nearrow T$. Due to uniform convergence we can fix $\delta$ and for all $\tau$ close enough to $T$ we get $\{S_\tau<R\}\subseteq\{f<R+\delta\}$, whence letting first $\tau\nearrow T$ and then $\delta\searrow 0$ we get 
\begin{equation*}
  \limsup_{\tau\nearrow T}\int_{\{S_\tau<R\}}a_\tau S_\tau\de x\leq M_{T,R}.  
\end{equation*}
Combining these inequalities into~\eqref{eq:almost_done1} gives~\eqref{eq:almost_exposed} and proves the result.
\end{proof}

To apply the preceding arguments to our situation, we need to extend a sequence as in Definition~\ref{def:pure_concentrations} to the whole space, for which we will use the following lemma.

\begin{lemma} \label{lem:extension_lemma}
Let $(u_j)_j$ be a sequence as in Definition~\ref{def:pure_concentrations} and assume that $|u_j|\toweakstar\lambda\in\mathcal{M}^+(\overline{Q})$ with
\[
  \lambda(\partial Q_s)=0.
\]
Let $\chi_s\in {\rm C}^\infty_c(Q;[0,1])$ be a smooth cutoff with $\chi_s\equiv 1$ on $Q_{s}$ and define
\[
  v_j:=\chi_su_j\in {\rm C}^\infty_c(\R^d;V).
\]
Then, there exists $C=C(\mathcal{A},s)$ such that 
\begin{equation}
    \label{eq:estimate1}
    \|v_j\|_{{\rm L}^1} + \|\mathcal{A}v_j\|_{{\rm L}^1} \leq C,\qquad j\in\N,
\end{equation}
and
\[
  v_j,\mathcal{A}v_j \to 0 \quad\text{in measure.}
\]
Moreover, with
\begin{equation}
    \label{eq:def_E_T}
    E_{j,T}:=\int_{Q_s}(1-a_T)|u_j|\de x+\int_{Q_s^c}a_T|v_j|\de x
\end{equation}
it holds that
\begin{equation}
    \label{eq:E_T_to_zero}
    \lim_{T\to 0}\limsup_{j\to\infty}E_{j,T}=0
\end{equation}
and
\begin{equation}
    \label{eq:crucial_ineq1}
    \bigg|\int_{\R^d}a_Tq(v_j)\de x-\int_{Q_s}q(u_j)\de x\bigg|\leq \Bigl(\max_{|z|=1}|q(z)|\Bigr) E_{j,T},
\end{equation}
for all $q:V\to\R$ which are continuous and positively $1$-homogeneous.
\end{lemma}
\begin{proof}
We have $|v_j|\leq |u_j|$ and $\mathcal{A}v_j=\sum_{k=1}^dA_ku_j \, \partial_k\chi_s$, whereby
\begin{equation*}
|\mathcal{A}v_j|\leq d\max_{k}\bigl\{|A_k| \cdot \|\partial_k\chi_s\|_{{\rm L}^\infty}\bigr\}|u_j|.
\end{equation*} 
From this, inequality~\eqref{eq:estimate1} follows, together with the convergence to zero in measure.

Now we can use the Gaussian estimates of the heat kernel to infer the existence of a dimensional constant $C>0$ such that
\begin{equation*}
a_T(x)\leq C{\rm e}^{-{\rm dist}(x,\partial Q_s)^2/(CT)}\;\text{for $x\notin Q_s$},\qquad 1-a_T(x)\leq C{\rm e}^{-{\rm dist}(x,\partial Q_s)^2/(CT)}\;\text{for $x\in Q_s$}.
\end{equation*}
Indeed, by classical Gaussian estimates we have for $x\in Q_s$ that
\begin{equation*}
    1-a_T(x)=\int_{Q_s^c}p_T(x-y)\de y\leq\int_{\{|z|\geq{\rm dist}(x,\partial Q_s)\}}p_T(z)\de z\leq C{\rm e}^{-{\rm dist}^2(x,\partial Q_s)/(CT)}
\end{equation*}
and similarly for $a_T$. We now apply the previous bounds in~\eqref{eq:def_E_T} to get
\begin{equation*}
   |E_{j,T}|\leq C\int_{Q}{\rm e}^{-{\rm dist}(x,\partial Q_s)^2/(CT)}|u_j(x)| \de x.
\end{equation*}
Taking the $\limsup$ as $j\to\infty$ and making use of the convergence of $|u_j|$ to $\lambda$, we get
\begin{equation*}
    \limsup_{j\to\infty} |E_{j,T}|\leq C\int_{Q}{\rm e}^{-{\rm dist}(x,\partial Q_s)^2/(CT)}\de\lambda(x).
\end{equation*}
Finally, we send $T\searrow 0$ and use that ${\rm e}^{-{\rm dist}(x,\partial Q_s)^2/(CT)}$ converges to $\ONE_{\partial Q_s}$ in conjunction with $\lambda(\partial Q_s)=0$ to arrive at~\eqref{eq:E_T_to_zero}.

Setting $C_q := \max_{|z|=1}|q(z)|$, we also have
\begin{equation*}
    \bigg|\int_{\R^d}a_Tq(v_j)\de x-\int_{Q_s}q(u_j)\de x\bigg|\leq C_q\int_{Q_s}(1-a_T)|u_j|\de x+C_q\int_{Q_s^c}a_T|v_j|\de x=C_qE_{j,T},
\end{equation*}
where we used the fact that $u_j=v_j$ on $Q_s$ and that $q$ is positively $1$-homogeneous. This is~\eqref{eq:crucial_ineq1}.
\end{proof}

We are now ready to establish the main result of this section.

\begin{proof}[Proof of Proposition~\ref{pro[:compensated_compactness_convex}]
Define 
\begin{equation*}
    f_{j,T}:=|v_j|+\sqrt{T}|\mathcal{A}v_j|,
\end{equation*}
where $(v_j)_j$ is as in Lemma~\ref{lem:extension_lemma}.
Thanks to~\eqref{eq:estimate1}, $\|f_{j,T}\|_{{\rm L}^1}\leq C$ uniformly in $j$ and moreover $f_{j,T} \to 0$ in measure as $j \to \infty$. The classical ${\rm L}^1$-to-${\rm L}^\infty$ estimate for the heat flow (H\"older's inequality) says that there exists $R_0>0$ (depending on $T$) such that
\begin{equation*}
    \sup_{j\in\N}\|\mathcal{P}_Tf_{j,T}\|_{{\rm L}^\infty}<R_0.
\end{equation*}
Then we have
\begin{equation*}
   \int_{\{f_{j,T}\leq R\}}a_Tf_{j,T}\de x\leq \int_{\{f_{j,T}\leq R\}}f_{j,T}\de x\to 0 \qquad\text{as $j\to\infty$.}
\end{equation*}
Potentially adjusting $R_0$ so that $\|S_0\|_{{\rm L}^\infty}=\|\mathcal{P}_Tf\|_{{\rm L}^\infty}<R_0$, we now apply Proposition~\ref{prop:almost_exposed}, obtaining the lower bound~\eqref{eq:almost_exposed}, in which we then send $j\to\infty$ for fixed $s,T,R,\eps$. This yields the estimate
\begin{equation*}
   \liminf_{j\to\infty}\int_{\R^d}a_Tq(v_j)\de x\geq-\eps\sup_jM_{j,T}-CN_q\frac{\sup_jM_{j,T}}{\log(R/R_0)}.
\end{equation*}
Now let $R\to\infty$ and $\eps\searrow 0$ to get 
\begin{equation*}
    \liminf_{j\to\infty}\int_{\R^d}a_Tq(v_j)\geq 0.
\end{equation*}
We then combine the latter with~\eqref{eq:crucial_ineq1} and~\eqref{eq:E_T_to_zero} and let $T \to 0$ to arrive at
\begin{equation} \label{eq:almost_lsc}
\liminf_{j\to\infty}\int_{Q_s}q(u_j)\de x\geq 0.
\end{equation}

On the other hand, using the convexity and the one-homogeneity of $q$, Jensen's inequality allows us to infer that
\begin{equation*}
   \int_{Q\setminus Q_s}q(u_j)\de x\geq q(u_j(\overline{Q}\setminus Q_s)).
\end{equation*}
Let us assume that $|u_j|\toweakstar\lambda\in\mathcal{M}^+(\overline{Q})$, which always holds up to selecting a subsequence (and if we select this first, the $\liminf$ is unaffected). For all but countably many $s\in(0,1/2)$ we have $\lambda(\partial Q_s)=0$. Choosing such values of $s$, from~\eqref{eq:almost_lsc} in conjunction with standard results in measure theory (see, e.g.,~\cite[Proposition 1.62]{AFP00}) we get
\begin{equation*}
  \liminf_{j\to\infty}\int_{Q}q(u_j)\de x\geq q(u(\overline{Q}\setminus Q_s)).
\end{equation*}
Finally, since $|u|(\partial Q)=0$ (see~(iv) in Definition~\ref{def:pure_concentrations}) and $q(0)=0$, we may take the limit as $s\searrow 0$ along these good values of $s$, concluding that
\[
  \liminf_{j\to\infty}\int_{Q}q(u_j)\de x\geq 0,
\]
from which the claim of Proposition~\ref{pro[:compensated_compactness_convex} follows.
\end{proof}

\section{Proofs of the main results and the applications} \label{sec:proofs}

\begin{proof}[Proof of Theorem~\ref{thm:vanishing_mass}]
Consider $\mu\in\mathcal{Y}_\mathcal{A}$ and let $q:V\to\R$ be any barrier function for $\Lambda_\mathcal{A}$, that is, $q$ is assumed to be convex, positively $1$-homogeneous, and non-negative on $\Lambda_\mathcal{A}$. Apply Proposition~\ref{pro[:compensated_compactness_convex} to obtain 
\begin{equation*}
  \int_{\mathbb{S}_V}q\de\mu
  \geq \liminf_{j\to\infty}\int_{Q}q(u_j)\de x
  \geq 0,
\end{equation*}
Thanks to Proposition~\ref{prop:convex_dicomoty}, applied with $\Gamma=\Lambda_{\mathcal{A}}$, we then get that 
\begin{equation*}
    \mu\in \overline{\rm co}^{w*}\, \mathcal{W}_{\Lambda_{\mathcal{A}}} = \overline{\rm co}^{w*}\big\{\mu\in\mathcal{M}^1(\mathbb{S}_V) \;:\; [\mu]\in\Lambda_{\mathcal{A}}\big\}.
\end{equation*}
Finally, applying Proposition~\ref{prop:Choquet_repr} yields the decomposition~\eqref{thm:vanishing_mass}.

The precise statement of Conjecture~\ref{conj:vmc} follows by disintegrating $\pi({\rm d} \sigma) \otimes \sigma$ with respect to the map $\sigma \mapsto [\sigma]$.
\end{proof}

\begin{proof}[Proof of Theorem~\ref{thm:vanishing_mass2}]
We already know from Proposition~\ref{prop:Young_convexity} that the set $\mathcal{Y}_{\mathcal{A}}$ is convex. Moreover, Theorem~\ref{thm:vanishing_mass} tells us that $\mathcal{Y}_\mathcal{A}\subseteq\overline{\rm co}^{w*}\, \mathcal{W}_{\Lambda_{\mathcal{A}}}$. In the following we will show also the opposite inclusion, which in particular implies the additional claim of the theorem. For this consider the subset $\mathcal{Y}_\mathcal{A}^{\rm unif}\subseteq\mathcal{Y}_\mathcal{A}$ of those Young measures $\mu$ generated by a sequence $(u_j)_j$ as in Definition~\ref{def:pure_concentrations} satisfying the additional constraint that $|u_j|\toweakstar\mathscr L^d \mres Q$ (a ``uniform'' concentration).

Let $\mu\in \mathcal{W}_{\Lambda_{\mathcal{A}}}$. We will show that $\mu\in\mathcal{Y}_\mathcal{A}^{\rm unif}$. Consider any $f:V\to\R$ that is $\Lambda_{\mathcal{A}}$-convex and has linear-growth, that is, $|f(\xi)| \leq C(1+|\xi|)$ for all $\xi \in V$, and such that its (strong) recession function
\[
  f^\infty(\xi) := \lim_{\substack{\!\!\!\! \xi' \to \xi \\ \; t \to \infty}} \frac{f(t\xi')}{t}
\]
exists. Under the present assumptions on the operator $\mathcal{A}$ we have that $f^\infty$ is convex at points in $\Lambda_{\mathcal{A}}$, see~\cite[Theorem 1.1]{KK16} (this is reproduced with a streamlined proof in~\cite[Theorem 13.17]{R26}). In particular, Jensen's inequality holds for any probability measure with barycenter in $\Lambda_{\mathcal{A}}$.

By the main result of~\cite{KR22} or~\cite{AR21} (for $\mathcal{A} = {\rm curl}$, this was already proved in \cite{KR10}), to prove that $\mu\in\mathcal{Y}_\mathcal{A}^{\rm unif}$ we are left to verify the generalized Jensen inequalities for the (generalized) Young measure $(\delta_0,\mu,\mathscr L^d\mres\Omega)$; see Chapter~15 in~\cite{R26} for the necessary terminology. Hence, we need to verify that 
\begin{equation*}
    f([\mu])\leq f(0)+\int_{\mathbb S_V} f^\infty\de\mu.
\end{equation*}
Shifting $f$ by an additive constant (which affects neither the $\Lambda_\mathcal{A}$-convexity nor the recession function), we may additionally require that $f(0) = 0$. Therefore, we just need to check the inequality
\begin{equation*}
    f([\mu])\leq\int_{\mathbb S_V} f^\infty\de\mu,
\end{equation*}
for those $f$. Now, for $t\geq 1$, we have
\begin{equation*}
   f([\mu])= f\bigl(t^{-1}(t[\mu])+(1-t^{-1})0\bigr)\leq\frac{1}{t}f(t[\mu]),
\end{equation*}
whence 
\begin{equation*}
   f([\mu])\leq f^\infty([\mu]).
\end{equation*}
Applying Jensen's inequality at $[\mu] \in \Lambda_{\mathcal{A}}$ we then get
\begin{equation*}
   f([\mu])\leq f^\infty([\mu])\leq\int_{\mathbb S_V} f^\infty\de\mu.
\end{equation*}
As this holds for all $f$ as above, the characterization result mentioned above then allows us to conclude that $\mu\in\mathcal{Y}_\mathcal{A}^{\rm unif}\subseteq\mathcal{Y}_\mathcal{A}$. Thus, $\overline{\rm co}^{w*}\, \mathcal{W}_{\Lambda_{\mathcal{A}}}\subseteq\mathcal{Y}_\mathcal{A}$ thanks to Lemma~\ref{lem:hom_weak_star_closed} and Proposition~\ref{prop:Young_convexity}, so that ultimately $\overline{\rm co}^{w*}\, \mathcal{W}_{\Lambda_{\mathcal{A}}} = \mathcal{Y}_\mathcal{A}$.
\end{proof}

\begin{proof}[Proof of Theorem~\ref{thm:giancarlo}]
As in the proof of~\cite[Theorem 1.1]{DPR16}, set 
\begin{equation*}
    P(x):=\frac{\de\mu}{\de|\mu|}(x).
\end{equation*}
Choose $x_0\in \Omega$ and a sequence $r_j\to 0$ such that 
\begin{enumerate}[(i)]
    \item $\displaystyle \lim_{j\to\infty}\frac{|\mu|^a(B_{r_j}(x_0))}{|\mu|^s(B_{r_j}(x_0))}=0.$
    \item $\displaystyle \lim_{j\to\infty}\frac{1}{|\mu|^s(B_{r_j}(x_0))}\int_{B_{r_j}(x_0)}|P-P(x_0)|\de|\mu|^s=0.$
    \item There exists a positive Radon measure $\nu\in\mathcal{M}(\R^d)$ such that $\nu \mres B_{1/2}(0) \neq 0$ and
    \begin{equation*}
        \nu_j:=\frac{T^{x_0,r_j}_\sharp|\mu|^s}{|\mu|^s(B_{r_j}(x_0))}\toweakstar\nu\quad{\rm as}\;j\to\infty,
    \end{equation*}
   where $T^{x_0,r_j}(x):=(x-x_0)/r_j$.
\end{enumerate}
As detailed in \emph{loc.\ cit.}, $|\mu|^s$-almost every $x_0$ satisfies these conditions. Now define 
\begin{equation*}
    \sigma_j:=\frac{T^{x_0,r_j}_\sharp\mu}{|\mu|^s(B_{r_j}(x_0))}
\end{equation*}
and observe that the homogeneity of $\mathcal{A}$ implies that 
\begin{equation}
\label{eq:fausto}
    \mathcal{A}\sigma_j=0\quad{\rm in}\;B_1.
\end{equation}
We then have, with $P_0 := P(x_0)$,
\begin{equation}
\label{eq:L1_close}
    |\sigma_j-P_0\nu_j|(B_1)\leq\frac{1}{|\mu|^s(B_{r_j}(x_0))}\int_{B_{r_j}(x_0)}|P-P_0|\de|\mu|^s+\frac{|\mu|^a(B_{r_j}(x_0))}{|\mu|^s(B_{r_j}(x_0))},
\end{equation}
and the right-hand side converges to zero as $j\to\infty$ due to the previous considerations.

Now consider a cube  $Q$ inside $B_1$ and without loss of generality assume that $ \nu(Q)>0$ and $\nu_j(\partial Q)=\nu(\partial Q)=0$ for all $j\in\N$. Then let $\rho_\eps$ be a standard mollifier and observe that due to the fact that $\nu_j$ is singular with respect to $\mathscr L^d$ the convolution $v_{j,\eps}:=\rho_\eps\ast\nu_j$ satisfies $v_{j,\eps} \to 0$ in measure on $Q$ as $\eps\to 0$. Moreover, by the properties of the mollification we have that for every $j \in \N$,
\begin{equation*}
    \|v_{j,\eps}\|_{{\rm L}^1(Q)}\to|\nu_j|(Q),\qquad \text{as $\eps \searrow 0$.}
\end{equation*}
We can therefore choose a sequence $\eps_j\searrow 0$ such that the sequence $v_j:=\rho_{\eps_j}\ast\nu_j$, $j \in \N$, converges to zero in measure in $Q$ and converges strictly to $\nu$ in $Q$ (see~\cite[Lemmas~13.4, 13.6]{R26} for the details). Consider the sequence $u_j := \rho_{\eps_j}\ast\sigma_j$, $j \in \N$. Thanks to~\eqref{eq:fausto},~\eqref{eq:L1_close} and the properties of $v_j$ we get that $u_j:Q\to V$  (once normalized by its ${\rm L}^1$-norm, which we assume without loss of generality) is admissible in Definition~\ref{def:pure_concentrations}. Using the Reshetnyak Continuity Theorem (see~\cite[Theorem~13.3]{R26}), we therefore get that for all $\psi\in {\rm C}(\mathbb S^{m-1})$ it holds that
\begin{equation*}
    \int_{\mathbb S^{m-1}}\psi\de\Theta_{u_j}\to\psi(P_0),
\end{equation*}
implying that $\Theta_{u_j}\toweakstar\mu=\delta_{P_0}$ in $\mathcal{M}^1(\mathbb S^{m-1})$.

Applying Theorem~\ref{thm:vanishing_mass} we get that there exists $\pi\in\mathcal{M}^1(\mathcal{W}_{\Lambda_{\mathcal{A}}})$ such that 
\begin{equation*}
\delta_{P_0}=\int_{\mathcal{W}_{\Lambda_{\mathcal{A}}}}\nu\de\pi(\nu).
\end{equation*}
But the only possibility for such a decomposition is that $\pi=\delta_{\delta_{P_0}}$, with $[\delta_{P_0}]=P_0\in\Lambda_{\mathcal{A}}$, thus proving the result.
\end{proof}

\begin{proof}[Proof of Theorem~\ref{thm:cone}]
We start by observing that there exists a direction $e_0\in\mathbb S_V$ and a constant $\gamma>0$ such that 
\begin{equation*}
    \langle y,e_0\rangle\geq\gamma|y|,\qquad y\in K.
\end{equation*}
For a general $z\in V$ we may thus choose $y=y(z)\in K$ such that 
\begin{equation*}
    |y-z|={\rm dist}(z,K),
\end{equation*}
and estimate as follows:
\begin{equation}
\label{eq:crucial_geo_ineq}
   \langle z,e_0\rangle=\langle y,e_0\rangle+\langle z-y,e_0\rangle\geq\gamma|y|-{\rm dist}(z,K)\geq\gamma|z|-(1+\gamma){\rm dist}(z,K),
\end{equation}
where we also used that $|y|\geq |z|-|y-z|$.

Suppose that $\Omega^\prime\subseteq U\subseteq\Omega$. Then we claim that
\begin{equation}
\label{eq:higher_int_claim}
\lim_{\delta\to 0}\limsup_{j\to\infty}\sup_{E\subseteq\Omega^\prime,\mathscr L^d(E)\leq\delta}|\mu_j|(E)=0.
\end{equation}
Indeed, if this were not true, there would exist a constant $\eps>0$, a subsequence of $j$'s (not relabeled), and compact sets $E_j\subseteq\Omega^\prime$ such that $\mathscr{L}^d(E_j)\to 0$ as $j\to\infty$, but for which $|\mu_j|(E_j)\geq \eps$. Then pick functions $\varphi_j\in {\rm C}_c^\infty(U)$ such that $\varphi_j\equiv 1$ on $E_j$, $0\leq\varphi_j\leq 1$, and $\mathscr L^d(\{\varphi_j>0\})\to 0$ as $j\to\infty$ (one can do this for example by slightly enlarging $E_j$ and mollifying $\ONE_{E_j}$).

We now choose a sequence $\eps_j\searrow 0$ such that for $u_j:=\mu_j\ast\rho_{\eps_j}$ (where $(\rho_\eps)_{\eps>0}$ is a standard mollifier family) it holds that 
\begin{equation}
\label{eq:strict_closedness}
    \int_U\varphi_j\, |\langle e_0,u_j\rangle|\de x\geq\int_{U}\varphi_j\biggl|\biggl\langle e_0,\frac{\de\mu_j}{\de|\mu_j|}\biggr\rangle\biggr|\de|\mu_j|-\frac{1}{j}
\end{equation}
and $\sup_{j\in\N}\|u_j\|_{{\rm L}^1(U)}<+\infty$. This is possible because the mollification of a measure converges strictly to it, as the mollification parameter tends to zero. Moreover, for $j$ large enough we have $\mathcal{A}u_j=0$ in $U$.

Additionally, since the map $z\mapsto{\rm dist}(z,K)$ is convex (as $K$ is a convex set) and $1$-homogeneous, we also get for $x\in U$ (applying Jensen's and Young's convolution inequality) that
\begin{align*}
    \|{\rm dist}(u_j,K)\|_{{\rm L}^1(U)}
    &\leq\int_{\Omega}\rho_{\eps_j}(x-y)\, {\rm dist}\biggl(\frac{\de\mu_j}{\de|\mu_j|}(y),K\biggr)\de|\mu_j|(y) \\
    &\leq \|\rho_{\eps_j}\|_{{\rm L}^1} \int_{\Omega}{\rm dist}\biggl(\frac{\de\mu_j}{\de|\mu_j|}(y),K\biggr)\de|\mu_j|(y)
\end{align*}
whence ${\rm dist}(u_j,K)\to 0$ in ${\rm L}^1(U)$ thanks to~\eqref{eq:distance_from_K}.

Assume now by contradiction that the sequence $(u_j)_j$ defined above is not equi-integrable in $\Omega^\prime$. Then by Lemma~\ref{lem:concentration_extraction} below there exists $\nu\in\mathcal{Y}_\mathcal{A}$ such that ${\rm supp}\, \nu\subseteq K\cap\mathbb S_V$. Indeed, the latter inclusion follows by choosing $q(z):={\rm dist}(z,K)$ in~\eqref{eq:67}. We can now apply Theorem~\ref{thm:vanishing_mass} to find a $\pi\in\mathcal{M}^1(\mathcal{W}_{\Lambda_{\mathcal{A}}})$ with
\begin{equation*}
\nu=\int_{\mathcal{W}_{\Lambda_{\mathcal{A}}}}\sigma\de\pi(\sigma).
\end{equation*}
Combining the previous facts, we get
\begin{equation*}
    0=\nu(\mathbb S_V\setminus K)=\int_{\mathcal{W}_{\Lambda_{\mathcal{A}}}}\sigma(\mathbb S_V\setminus K)\de\pi(\sigma).
\end{equation*}
Thus, ${\rm supp}\, \sigma\subseteq K\cap\mathbb S_V$ for $\pi$-almost every $\sigma$. Then, the convexity of $K$ forces $[\sigma]\in K$, while at the same time one has $[\sigma]\in\Lambda_{\mathcal{A}}$. Via the assumption that $K \cap \Lambda_{\mathcal{A}} = \{0\}$, the only possibility for both of these conditions to be true at the same time is $[\sigma]=0$. On the other hand we have, by definition of $K$,
\begin{equation*}
    \langle e_0,[\sigma]\rangle=\int_{\mathbb{S}_V}\langle e_0,z\rangle\de\sigma(z)\geq\gamma>0,
\end{equation*}
which is a contradiction. Therefore, the sequence $(u_j)_j$ is equi-integrable.

Applying~\eqref{eq:crucial_geo_ineq} to $\frac{\de\mu_j}{\de|\mu_j|}$, we then get 
\begin{equation*}
    \biggl|\biggl\langle e_0,\frac{\de\mu_j}{\de|\mu_j|}\biggr\rangle\biggr|
    \geq\gamma-(1+\gamma){\rm dist}\bigg(\frac{\de\mu_j}{\de|\mu_j|},K\bigg),
\end{equation*}
which together with~\eqref{eq:strict_closedness} gives (multiplying by $\varphi_j$ and integrating over $U$)
\begin{equation*}
     \int_U\varphi_j\, |\langle e_0,u_j\rangle|\de x\geq \gamma\int_U\varphi_j\de|\mu_j|-(1+\gamma) \int_U{\rm dist}\bigg(\frac{\de\mu_j}{\de|\mu_j|},K\bigg)\de|\mu_j|-\frac{1}{j},
\end{equation*}
for all $j\in\N$. Taking the lower limit as $j\to\infty$ and using~\eqref{eq:distance_from_K} together with
\begin{equation*}
    \liminf_{j\to\infty}\int_U\varphi_j\, |\langle e_0,u_j\rangle|\de x\leq\lim_{j\to\infty}\int_{\{\varphi_j>0\}}|u_j|\de x=0,
\end{equation*}
which follows from Vitali's convergence theorem due to the equi-integrability of the sequence $(u_j)_j$, leads to
\[
  0\geq\gamma\liminf_{j\to\infty}|\mu_j|(E_j)\geq\gamma\eps,
\]
a contradiction. Therefore, the claim~\eqref{eq:higher_int_claim} holds.

Now we decompose $\mu_j=g_j\mathcal L^d+\mu_j^s$ and set $N:=\bigcup_{j\in\N}N_j$, where $N_j\subseteq\Omega^\prime$ is such that $\mu_j^s$ is concentrated on $N_j$ and $\mathscr L^d(N_j)=0$. This gives $\mathscr L^d(N)=0$ and so from~\eqref{eq:higher_int_claim} we infer
\begin{equation*}
    |\mu_j|^s(\Omega^\prime)\leq|\mu_j|(N)\to 0\quad{\rm as}\;j\to\infty.
\end{equation*}
Moreover, we have
\begin{equation*}
    \int_{B}|g_j|\de x\leq|\mu_j|(B),\qquad \text{$B\subseteq\Omega^\prime$ Borel,}
\end{equation*}
so that the family $(g_j)_j$ is equi-integrable, again by the claim~\eqref{eq:higher_int_claim}.

The addition follows by applying Theorem~\ref{thm:singular_density} to the $\mathcal{A}$-free measure $\mu_j$ and the higher-integrability results of~\cite[Section~3]{S12}.
\end{proof}

We still need to prove the following lemma, which translates a result from~\cite{ARA26} into our language (see again~\cite[Chapter~15]{R26} for the background on generalized Young measures).

\begin{lemma}
\label{lem:concentration_extraction}
Let $\mathcal{A}$ be a first-order constant-coefficient linear differential operator with constant rank. Let $(u_j)_{j\in\N}\subseteq {\rm L}^1(\Omega;V)$ be uniformly bounded and $\mathcal{A}$-free. If $(u_j)_j$ is not equi-integrable on some $U\subseteq\Omega$ then there exists $\nu\in\mathcal{Y}_{\mathcal{A}}$ such that if $q:V\to[0,\infty)$ is continuous and positively $1$-homogeneous and if 
\begin{equation}
\label{eq:marcellinopaneevino}
    \lim_{j\to\infty}\int_{\Omega^\prime}q(u_j)\de x\to 0  \qquad\text{for all $\Omega' \subseteq U$,}
\end{equation}
then 
\begin{equation}
\label{eq:67}
    \int_{\mathbb S_V}q\de\nu=0.
\end{equation}
\end{lemma}
\begin{proof}
If $(u_j)_j$ as in the assumptions is not equi-integrable on some $U\subseteq\Omega$ then there exists non-zero concentration measure $\lambda$ and concentration-direction probability measures $(\nu_{x}^\infty)_{x\in\Omega}$ generated by $(u_j)_j$. Now choose $\Omega^\prime\subseteq U$ such that $\lambda(\partial\Omega^\prime)=0$ and $\lambda(\Omega^\prime)>0$. Set 
\begin{equation*}
    \nu:=\frac{1}{\lambda(\Omega^\prime)}\int_{\Omega^\prime}\nu_{x}^\infty\de\lambda(x).
\end{equation*}
The arguments in the proof of~\cite[Lemma 4.5]{ARA26} show that $\nu$ satisfies the homogeneous generalized Jensen inequality. Therefore, by~\cite[Proposition 4.4]{ARA26} we can find a sequence $(w_j)_j\subseteq {\rm L}^1(U;V)$ (a priori not smooth) such that $w_j \to 0$ in measure, $|w_j|\toweakstar\mathscr L^d\mres\Omega$, $w_j\toweakstar[\nu]\mathscr L^d\mres\Omega$, and $\Theta_{w_j}\toweakstar\nu$. After normalization and a mollification, we may without loss of generality assume that $(w_j)_j$ satisfies the properties listed in Definition~\ref{def:pure_concentrations} and hence $\nu\in\mathcal{Y}_{\mathcal{A}}$.

Now fix a $q:V\to[0,\infty)$ that is continuous and positively $1$-homogeneous and for which~\eqref{eq:marcellinopaneevino} holds. We then have
\begin{equation*}
     \int_{\mathbb S_V}q\de\nu=\frac{1}{\lambda(\Omega^\prime)}\int_{\Omega^\prime}\int_{\mathbb S_V}q\de\nu_x^\infty\de\lambda(x)=\lim_{j\to\infty}\int_{\Omega^\prime}q(w_j)\de x=0.
\end{equation*}
This concludes the proof.
\end{proof}

\begin{proof}[Proof of Theorem~\ref{thm:compensated_compactness}]
Let $A\in V$ and let $u_j$ be a sequence as in \cref{def:pure_concentrations}. Choose a sequence $M_j\uparrow\infty$ such that $\mathscr L^d(\{|u_j|\geq M_j\})\to 0$ as $j\to\infty$ and such that
\begin{equation*}
    \lim_{j\to\infty}\int_{\{|u_j|> M_j\}}|u_j|\de x=1.
\end{equation*}
We have 
\begin{equation*}
    \liminf_{j\to\infty}\int_Qf(A+u_j)\de x\geq\liminf_{j\to\infty}\int_{\{|u_j|\leq M_j\}}f(A+u_j)\de x+\liminf_{j\to\infty}\int_{\{|u_j|>M_j\}}f(A+u_j)\de x.
\end{equation*}
Thanks to the previous assumption, the first lower limit on the right-hand side is equal to $f(A)$. On the other hand, there exists $\mu\in\mathcal{Y}_\mathcal{A}$ such that 
\begin{equation*}
  \liminf_{j\to\infty}\int_Qf(A+u_j)\de x\geq f(A)+\int_{\mathbb S_V}f^\infty\de\mu\geq f(A)+\int\int_{\mathbb S_V} f^\infty\de\nu\de\pi\geq f(A),  
\end{equation*}
where we have used \cref{thm:vanishing_mass}, together with Jensen's inequality, the $\Lambda_{\mathcal{A}}$-convexity of $f$, and the non-negativity of the recession function at points in $\Lambda_{\mathcal{A}}$ (which follows from the corresponding positivity property of $f$).
\end{proof}

\begin{proof}[Proof of Theorem~\ref{thm:extremals}]
The refined decomposition follows by replacing the application of Proposition~\ref{prop:Choquet_repr} in the proof of Theorem~\ref{thm:vanishing_mass} with an invocation of the full Choquet theorem (see, e.g.,~\cite{P01}).

Consider $\nu\in{\rm ex}\, \mathcal{W}_{\Lambda_{\mathcal{A}}}$. Then there exists $\xi\neq 0$ such that $[\nu]\in{\rm ker}\mathbb A(\xi)=:L$. Since $L\subseteq\Lambda_{\mathcal{A}}$ we have that 
$\mathcal{W}_{L}:=\{\nu\in\mathcal{M}^1(\mathbb S_V) \;:\; [\nu]\in L\}$ is contained in $\mathcal{W}_{\Lambda_{\mathcal{A}}}$. Therefore, $\nu$ is also extremal in $\mathcal{W}_L$. If, by contradiction, $\#{\rm supp}\, \nu>\ell+1$ where $\ell:={\rm rank}\, \mathbb A(\xi)$, then there exist distinct $z_1,\dots,z_{\ell+2}\in{\rm supp}\, \nu$. Set $r := (\min_{i\neq j}|z_i-z_j|)/3$ and define the disjoint family of sets $E_i:=B_r(z_i)\cap\mathbb S_V$ for $i=1,\dots,\ell+2$, all of which satisfy $\nu(E_i) > 0$, as well as the vectors
\begin{equation*}
    v_i=\bigg(\nu(E_i),\int_{E_i}P_{L^\perp}[z]\de\nu(z)\bigg)\in\R\times L^\perp,
\end{equation*}
where $P_{L^\perp}:\mathbb S_V\to L^\perp$ denotes the orthogonal projection onto $L^\perp$. Since $\ell={\rm rank}\, A(\xi)={\rm dim}\, L^\perp$, and hence ${\rm dim}\, (\R\times L^\perp)=\ell+1$, the vectors $v_1,\dots,v_{\ell+2}\in\R\times L^\perp$ must be linearly dependent. Thus we find scalars $a_i\in\R$, not all of which are zero, such that 
\begin{equation*}
    \sum_{i =1}^{\ell+2}a_i\nu(E_i)=0,\qquad
    \sum_{i =1}^{\ell+2}a_i\int_{E_i}P_{L^\perp}[z]\de\nu(z)=0.
\end{equation*}
Set now $M:=\max_{i=1,\dots,\ell+2}|a_i|$ and define $g:\mathbb S_V\to\R$ via 
\begin{equation*}
g(z):=\frac{1}{2M}\sum_{i=1}^{\ell+2}a_i\ONE_{E_i},  \qquad z \in \mathbb S_V.
\end{equation*}
Then, $|g|\leq 1/2$ and 
\begin{equation*}
    \int g\de\nu=0, \qquad
    \int_{\mathbb S_V}g(z)P_{L^\perp}[z]\de\nu(z)=0.
\end{equation*}
The latter properties imply that the measures
\[
  \nu^{\pm}:=(1\pm g)\nu
\]
are both well-defined probability measures on $\mathbb{S}_V$ and they are such that $\nu^+\neq\nu^-$ (since $g\neq 0$ on a set of positive measure). Moreover, we have
\begin{equation*}
    P_{L^\perp}[\nu^{\pm}]
    =P_{L^\perp}[\nu]\pm\int_{\mathbb S_V}g(z)P_{L^\perp}[z]\de\nu(z)=0
\end{equation*}
and thus $[\nu^\pm] \in L \subseteq \Lambda_{\mathcal{A}}$. Finally, $\nu=(\nu^++\nu^-)/2$, implying that $\nu$ is not extremal. This is a contradiction to our assumption above.
\end{proof}

\appendix
\section{Appendix: Examples of concentrating sequences}\label{sec:examples}

In this appendix we present some examples of measures on $\R^{2 \times 2}$ that lie in $\mathcal{Y}_{\rm div}$, and their generating sequences $u_j:(0,1)^2\subseteq\R^2\to\R^{2\times2}$, $j \in \N$, which are row-wise divergence-free. Technically speaking, the sequences built here are not admissible in Definition~\ref{def:pure_concentrations} as they will often be merely piecewise constant and not smooth. However, via a mollification, which we do not write out, it is possible to make them admissible while preserving their significant properties.

\medskip
\textit{Example 1.}\;
The first example concerns a sequence $(u_j)_j$ converging weakly* and strictly to the measure 
\begin{equation*}
    u=\begin{pmatrix}
    1 & 0 \\
    0 & 0
    \end{pmatrix}\mathscr L^2 \mres Q=: {\rm e}_1\otimes {\rm e}_1\, \mathscr L^2 \mres Q,
\end{equation*}
and whose generated concentration Young measure $\mu$ equals $\delta_{{\rm e}_1\otimes {\rm e}_1}$. To construct it, consider $S_j\subseteq(0,1)$ to be the union of disjoint open intervals with total length $\rho_j:=\mathscr L^1(S_j)\to 0$ as $j \to \infty$, and such that the intervals are equi-distributed in the sense that for every $\phi\in {\rm C}([0,1])$ it holds that 
\begin{equation*}
    \frac{1}{\rho_j}\int_{S_j}\phi\de t\to\int_0^1\phi\de t.
\end{equation*}
For instance, we could take on the order of $\sqrt{j}$ intervals of length about $1/j$ (so that $\rho_j \sim 1/\sqrt{j}$). Then, defining the set $A_j:=(0,1)\times S_j\subseteq\R^2$, let
\begin{equation}
    \label{ex:first}
    u_j:= \frac{1}{\rho_j}\begin{pmatrix}
        1  & 0 \\
        0 & 0
    \end{pmatrix}\ONE_{A_j}.
\end{equation}
It is easy to verify that these vector fields satisfy all the assumptions in Definition~\ref{def:pure_concentrations}, apart from smoothness, so that $\mu = \delta_{{\rm e}_1\otimes {\rm e}_1}$ belongs to $\mathcal{Y}_\mathcal{\rm div}$; in fact, $\mu$ even belongs to $\mathcal{Y}_{\rm div}^{\rm unif}$. Moreover, ${\rm e}_1\otimes {\rm e}_1 \in\Lambda_{\rm div}$ as ${\rm det}\, {\rm e}_1\otimes {\rm e}_1=0$.

\medskip
\textit{Example 2.}\;
Defining the matrices $E_{11}:= {\rm e}_1\otimes {\rm e}_1$ and $E_{22}:={\rm e}_2\otimes {\rm e}_2$, the second example we consider will generate the concentration Young measure 
\begin{equation*}
\mu:=\frac{1}{2}\delta_{E_{11}}+\frac{1}{2}\delta_{E_{22}}.  
\end{equation*}
For this, consider the sets $S_j$ defined in the previous example and set
\[
  A_j:=(0,1)\times S_j, \qquad
  B_j:=S_j\times(0,1).
\]
Then let
\begin{equation*}
 v_j:=\frac{1}{\rho_j}\begin{pmatrix}
        1 & 0 \\
        0 & 0
    \end{pmatrix}\ONE_{A_j}
    +\frac{1}{\rho_j}\begin{pmatrix}
        0 & 0 \\
        0 & 1
    \end{pmatrix}\ONE_{B_j},\qquad
    u_j:=\frac{v_j}{\|v_j\|_{{\rm L}^1}}.
\end{equation*}
It is clear that the sequence $(u_j)_j$ has all the desired properties, apart from the smoothness. On the other hand, let us stress that $(u_j)_j$ does not converge strictly to the limit (otherwise the generated measure $\mu$ would be a single Dirac mass). Observe that $E_{11}, E_{22}\in\Lambda_{\rm div}$, so a decomposition as in Theorems~\ref{thm:vanishing_mass},~\ref{thm:vanishing_mass2} is possible for this $\mu$.

\medskip
Before we move on, let us mention that in the appendix to~\cite{BIR23} a much more sophisticated example in the same spirit as the two above is presented in order to illustrate the upper bound in the Optimal Light Structures Conjecture.

\medskip
\textit{Example 3.}\;
The last example, which can be found already in~\cite{B03} (where it is attributed to Seppecher) will show that in the statement of Conjecture~\ref{conj:vmc} one cannot ask for the measures $\nu$ of the decomposition to satisfy the stronger condition ${\rm supp}\, \nu\subseteq\Lambda_{\rm div}$. For this, for every $j$ sufficiently large denote by $A_j$ an ensemble of disjoint copies of the open ball $B_{r_j}$, with radius $r_j$ chosen such that $\pi r_j^2=1/(4j)$, which are equidistantly arranged in the square $Q$ in a way that
\[
  4\sqrt{j}\int_{A_j}\phi\de x \to \int_Q \phi\de x
\]
for every $\phi\in {\rm C}([0,1]^2)$. Loosely speaking, one has that $A_j$ consists of $\sqrt{j}$ balls of area $1/(4j)$ each. This also implies that $\mathscr L^2(A_j)\to 0$ as $j\to\infty$. One can then find for every $j \in \N$ a smooth and compactly supported function $v_j\in {\rm C}_c^\infty(Q;\R^{2\times 2})$ such that ${\rm div}\,v_j=0$ and $v_j = 4\sqrt{j}\, {\rm Id}$ on $A_j$, where ${\rm Id}$ denotes the identity matrix. The normalized sequence $u_j:=\|v_j\|_{{\rm L}^1}^{-1}\, v_j$ then satisfies the assumptions of Definition~\ref{def:pure_concentrations} and it is not hard to see that the generated concentration Young measure $\mu$ contains a Dirac mass centered at the identity matrix. Thus, ${\rm supp}\, \mu$ is not a subset of $\Lambda_{\rm div}$. Interestingly, in this example, $u_j \toweakstar 0$ (test with $\phi(x) := x_k$ for $k = 1,2$ and apply the divergence theorem). Consequently, $[\mu] =0 \in \Lambda_{\rm div}$, so the decomposition of Theorems~\ref{thm:vanishing_mass},~\ref{thm:vanishing_mass2} is trivially satisfied.

\medskip

As a final remark, consider the measure $\mu\in\mathcal{M}^1(\mathbb S^3)$ (where $\mathbb S_{\R^{2 \times 2}}$ is the Frobenius-norm sphere in $\R^{2 \times 2}$) given by
\begin{equation*}
    \mu=\frac{1}{2}\delta_{{\rm Id}/\sqrt{2}}+\frac{1}{2}\delta_{-{\rm Id}/\sqrt{2}}.
\end{equation*}
Using again the characterization of generalized Young measures in~\cite{KR10,AR21,KR22} in conjunction with~\cite{KK16} it is possible to prove that $\mu$ belongs to $\mathcal{Y}_{\rm curl}^{\rm unif}$. However, it would be instructive to build an \emph{explicit} curl-free sequence $(u_j)_j$ that generates $\mu$ in the sense of Definition~\ref{def:pure_concentrations}.

\bibliography{references}
            \bibliographystyle{alpha}

\end{document}